\documentclass[a4paper,11pt]{amsart}
\usepackage{etex}
\usepackage{tikz, tikz-cd}

\usepackage{bm}
\usepackage{amssymb}
\usepackage{mathrsfs}
\usepackage{amsmath, amssymb,amsthm,latexsym,amscd,mathrsfs}
\usepackage{indentfirst}
\usepackage{stmaryrd}
\usepackage{graphicx}
\usepackage{subfigure}
\usepackage{extarrows}
\usepackage{amssymb}
\usepackage{amsmath,amssymb,amsthm,latexsym,amscd,mathrsfs}
\usepackage{indentfirst}
\usepackage{multirow}
\usepackage{latexsym}
\usepackage{amsfonts}
\usepackage{color}
\usepackage{pictexwd,dcpic}
\usepackage{graphicx}
\usepackage{psfrag}
\usepackage{hyperref}
\usepackage{comment}
\usepackage{cases}
\usepackage{dutchcal}

\makeatletter
\@namedef{subjclassname@2010}{{\rm2020} Mathematics Subject Classification}
\makeatother

\newcommand{\ROM}[1]{\mathrm{\uppercase\expandafter{\romannumeral#1}}}
\theoremstyle{definition}
\numberwithin{equation}{section}
\theoremstyle{plain}

\newtheorem{thm}{Theorem}
\newtheorem{lem}{Lemma}[section]
\newtheorem{defn}{Definition}[section]
\newtheorem{cor}{Corollary}

\newtheorem{rem}{Remark}
\newtheorem{prop}{Proposition}
\newtheorem{ques}{Question}[section]

\makeatletter

\allowdisplaybreaks
\makeatother
\title[Isoparametric foliations and complex structures]{\textbf{Isoparametric foliations and complex structures}}
\author[Chao Qian]{Chao Qian}\address{School of Mathematics and Statistics, Beijing Institute of Technology, Beijing
	100081, P.R. China}
\email{6120150035@bit.edu.cn}
\author[Zizhou Tang]{Zizhou Tang}\address{Chern Institute of Mathematics $\&$ LPMC, Nankai University, Tianjin 300071, P. R. China}
\email{zztang@nankai.edu.cn}
\author[Wenjiao Yan]{Wenjiao Yan}\address{School of Mathematical Sciences, Laboratory of Mathematics and Complex Systems, Beijing Normal University, Beijing, 100875, P. R. China}
\email{wjyan@bnu.edu.cn}
\thanks {$^{\dag}$ the corresponding author}
\thanks {The project is partially supported by the NSFC (No. 12271038, 12371048), Nankai Zhide Foundation, and the Fundamental Research Funds for the Central Universities (2233300002 )}


\begin{document}

\maketitle

\begin{abstract}
	Explicit representations of complex structures on closed manifolds are valuable, but relatively rare in the literature. Using isoparametric theory,
	we construct complex structures on isoparametric hypersurfaces with $g=4, m=1$ in the unit sphere, as well as on focal submanifolds $M_+$ of OT-FKM type with $g=4$, $m=2$ and $m=4$ in the definite case. Furthermore, we investigate the existence and non-existence of invariant (almost) complex structures on homogeneous isoparametric hypersurfaces with $g=4$, providing a complete classification of those that admit such structures. Finally, we discuss the geometric properties of the complex structures arising from isoparametric theory.

\end{abstract}

\section{\textbf{Introduction}}\label{sec1}

Determining whether a given closed differentiable manifold admits a complex structure is a challenging problem, as the Cauchy-Riemann equation becomes over-determined in higher dimensions.
Known examples of complex manifolds include complex submanifolds in complex projective spaces, Hopf manifolds, Calabi-Eckmann manifolds (\cite{CE53}), the product of two odd-dimensional Brieskorn homotopy spheres (\cite{BV68}), and twistor spaces of C. Taubes' self-dual 4-manifolds. Beyond these, few examples of complex manifolds are known.

To study the existence of complex structures on a given manifold, one first introduces the concept of an almost complex structure and then explores its integrability. An almost complex structure $J$ on a differentiable manifold
$M$ is an endomorphism of the tangent bundle, $J: TM\rightarrow TM$ with $J^2=-Id$. A manifold equipped with such a structure is called almost complex. 
It is clear that an almost complex manifold must be both even-dimensional and orientable. An orientable manifold $M^{2n}$ admits an almost complex structure if and only if the structure group of $TM$ reduces to the unitary group $U(n)$, making the existence of an almost complex structure a topological question. For example, Borel and Serre \cite{BS53} showed that
$S^2$ and $S^6$ are the only spheres admitting almost complex structures.
The next question is whether an almost complex structure on $M$ is induced by a complex structure. If so, the almost complex structure is called integrable, or simply a complex structure.
The Newlander-Nirenberg theorem asserts that an almost complex structure $J$ is integrable if and only if the corresponding Nijenhuis tensor $N(X, Y)$ vanishes everywhere, where
\begin{equation}\label{N}
	N(X, Y)=[JX, JY]-[X, Y]-J[JX, Y]-J[X, JY], \quad X, Y\in\mathfrak{X}(M)
\end{equation}
Unfortunately, despite having the Newlander-Nirenberg theorem, it remains unknown whether there exists a complex structure on $S^6$ (cf. \cite{Bla53}, \cite{LeB87}, \cite{BH99}, \cite{Tan06}, \cite{TY24}).

An intriguing problem that has captivated geometers and topologists for decades is finding the obstructions to the existence of complex structures on almost complex manifolds. In real dimension $4$, such obstructions are understood through the Enriques-Kodaira classification of complex surfaces, the Chern number inequality, and Seiberg-Witten invariants from gauge theory. In this case, there are many examples of $4$-dimensional closed almost complex manifolds that do not admit complex structures. 
For higher dimensions, S. T. Yau conjectured in his ``Open problems in Geometry" ( as the 52nd problem ) that any closed almost complex manifold with dimension greater than 4 admits a complex structure, though there is limited evidence supporting this conjecture. It is worth noting that Tang and Yan \cite{TY22} proved the existence of an 8-dimensional closed manifold $N^8$ such that $S^6\times N^8$ admits a complex structure, regardless of the existence of a complex structure on $S^6$.

In this paper, we explore the interaction between (almost) complex structures and isoparametric theory, focusing on the existence and construction of certain (almost) complex structures on isoparametric hypersurfaces and focal submanifolds in unit spheres.

We begin with a brief overview of isoparametric theory. Isoparametric hypersurfaces in unit spheres are those with constant principal curvatures, making them among the simplest and most elegant hypersurfaces, rich in geometric and topological information.
The foundational work in this theory was contributed by E. Cartan and H.-F. M\"{u}nzner. Let $M^n$ be an isoparametric hypersurface in $S^{n+1}(1)$ with $g$ distict principal curvatures $\cot{\theta_k}$, $0<\theta_{1}<\cdots<\theta_{g}<\pi$, and corresponding multiplicities $m_{k}$. According to M\"{u}nzner \cite{Mun81}, the relations $\theta_{k}=\theta_{1}+\frac{k-1}{g}\pi$, $m_{k}=m_{k+2}$(subscripts mod $g$) hold, and
$g$ can only take values in $1, 2,3, 4$ or $6$. The classification of isoparametric hypersurfaces in unit spheres, achieved through the efforts of numerous topologists and geometers, was completed in 2020 (see \cite{CCJ07}, \cite{Imm08}, \cite{Miy13}, \cite{Chi13}, \cite{Chi20}). Specifically, these hypersurfaces are either homogeneous or of OT-FKM type. Among them, those with four distinct principal curvatures, especially the OT-FKM type, exhibit the richest topology and geometry (see \cite{TY13}, \cite{TY15}, \cite{QT16}, \cite{QTY23}).

Now we introduce the isoparametric family of OT-FKM type.
Given a symmetric Clifford system $\{P_0,\cdots, P_m\}$ on $\mathbb{R}^{2l}$, i.e., $P_{i}$'s are symmetric matrices satisfying $P_{i}P_{j}+P_{j}P_{i}=2\delta_{ij}I_{2l}$, the following homogeneous polynomial $F$ of degree $4$ on
$\mathbb{R}^{2l}$ is constructed (cf. \cite{FKM81}):
\begin{equation}\label{FKM isop. poly.}
	\begin{split}
		&\qquad F:\quad \mathbb{R}^{2l}\longrightarrow \mathbb{R}\\
		&F(x) = |x|^4 - 2\sum_{i = 0}^{m}{\langle
			P_{i}x,x\rangle^2}.
	\end{split}
\end{equation}
The restricted function $f=F|_{S^{2l-1}(1)}$ takes values in $\mathrm{Im}f=[-1,1]$.
A regular level set $M^{2l-2}$ of $f$ is called \emph{an isoparametric hypersurface of OT-FKM type}, with $4$ dinstinct principal curvatures. The singular level sets, $M_+=f^{-1}(1)$ and $M_-=f^{-1}(-1)$, are called \emph{focal submanifolds}. The multiplicities of the principal curvatures of $M^{2l-2}$ are $(m_1, m_2, m_1, m_2)$, where $(m_1, m_2)=(m, l-m-1)$ and
$l=k\delta(m)$. Here, $k$ is a positive integer such that $l-m-1>0$ and $\delta(m)$ is the dimension of the irreducible module of the Clifford algebra $\mathcal{C}_{m-1}$, as shown in the following table:
\begin{center}
	\begin{tabular}{|c|c|c|c|c|c|c|c|c|cc|}
		\hline
		$m$ & 1 & 2 & 3 & 4 & 5 & 6 & 7 & 8 & $\cdots$ &$m'$+8 \\
		\hline
		$\delta(m)$ & 1 & 2 & 4 & 4 & 8 & 8 & 8 & 8 &$\cdots$ &16$\delta(m')$\\
		\hline
	\end{tabular} .
\end{center}
According to \cite{FKM81}, when $m\not \equiv 0 ~( \mathrm{mod}~4)$, there is exactly one type of OT-FKM type isoparametric family;
when $m\equiv 0~(\mathrm{mod} ~4)$, there are two types, distinguished by $\mathrm{Trace}(P_0P_1\cdots P_m)$: the \emph{definite} family, where $P_0P_1\cdots P_m=\pm I_{2l}$, and the \emph{indefinite} family, where $P_0P_1\cdots P_m\neq\pm I_{2l}$. There are exactly $[\frac{k}{2}]$ non-congruent indefinite families.
Moreover, an isoparametric hypersurface
$M^{2l-2}$ of OT-FKM type is an $S^{m_1}$ $(S^{m_2})$-bundle over $M_+$ $(M_-)$, and $M_+$ is an $S^{l-m-1}$-bundle over $S^{l-1}$. Specifically, the normal bundle of $M_+$ in $S^{2l-1}$ is trivial, with $\{P_0x,\cdots, P_mx\}$ serving as a basis for the normal sections at $x\in M_+$. Consequently, $M^{2l-2}$ is diffeomorphic to $S^m\times M_+$. (cf. \cite{Wan88}, \cite{QT16}).


In this paper, we first clarify the existence of almost complex structures on isoparametric hypersurfaces and focal submanifolds of OT-FKM type:

\begin{prop}\label{almost complex}
	In an isoparametric family with $g=4$ in $S^{n+1}(1)$, the following results hold:
	\begin{enumerate}
		\item Every isoparametric hypersurface admits an almost complex structure.
		\item Every even-dimensional focal submanifold $M_+$ of OT-FKM type admits an almost complex structure.
		\item An even-dimensional focal submanifold $M_-$ of OT-FKM type admits an almost complex structure if and only if $M_-$ is diffeomorphic to a product of two odd-dimensional spheres.
	\end{enumerate}
\end{prop}

By part (iii) of the preceding proposition, the focal submanifold $M_-$ is almost complex if and only if  it is a Calabi-Eckmann manifold, which admits a complex structure (cf. \cite{CE53}). As a consequence, we only need to focus on the existence of complex structures on isoparametric hypersurfaces $M$ and focal submanifolds $M_+$ of OT-FKM type. As the isoparametric hypersurfaces with $g=4, \min\{m_1, m_2\}=1$ are of OT-FKM type (cf. \cite{Tak76}, \cite{CCJ07}, \cite{Imm08}), we establish the following theorem by generalizing Cartan's moving frame method \cite{Car40}:

\begin{thm}\label{m=1}
	For $(g, m_1, m_2)=(4, 1, l-2)$, the isoparametric hypersurface $M$ admits a complex structure $J:TM\rightarrow TM$ such that $J{\mathcal{D}}_1={\mathcal{D}}_3$ and $J{\mathcal{D}}_2={\mathcal{D}}_4$, where $\mathcal{D}_i$s are principal distributions corresponding to the principal curvatures $\cot \theta_i$, 
	respectively.
\end{thm}

Using the algebraic structure of symmetric Clifford system, we construct explicit integrable almost complex structures to show
\begin{thm}\label{M+ complex}
	The focal submanifold $M_+$ of OT-FKM type in the following cases admit complex structures:
	\begin{enumerate}
		\item $M_+$ of OT-FKM type with $m=2$. Consequently, the isoparametric hypersurface $M\cong M_+\times S^2$ is a complex manifold.
		\item $M_+$ of OT-FKM type with $m=4$ in the definite case.
	\end{enumerate}
	
\end{thm}


Next, we turn to the homogeneous isoparametric family and examine the existence of invariant (almost) complex structures. For focal submanifold $M_+$ of OT-FKM type with $m=8$ in the definite case, which is the octonionic Stiefel manifold $V_2(\mathbb{O}^n)$ (\cite{Jam58}, \cite{QTY22}), we show the following

\begin{prop}\label{m=8 focal}
	The focal submanifold $M_+$ with $(g, m_1, m_2)=(4, 8, 7)$ of OT-FKM type in the definite case admits a transitive action by $Spin(9)$ and is diffeomorphic to $Spin(9)/G_2$. Moreover, $M_+$ does not admit any $Spin(9)$-invariant almost complex structure.
\end{prop}

\begin{rem}
	By Proposition \ref{m=8 focal}, it seems impossible to construct complex structures on focal submanifolds $M_+$ of OT-FKM type for $m\geq 8$ by using similar arguments as in Section \ref{M_+ with m=2,4} based on algebraic structure of symmetric Clifford systems.
\end{rem}

Moreover, we get

\begin{thm}\label{homogeneous g=4}
	For homogeneous isoparametric hypersurfaces with $g=4$,
	\begin{enumerate}
		\item When $(m_1, m_2)=(4, 4k-5)$ ($k\geq 2$), each homogeneous isoparametric hypersurface $M^{8k-2}$ corresponding to the symmetric pair $(Sp(k+2), Sp(2)\times Sp(k))$
		admits no $Sp(2)\times Sp(k)$-invariant almost complex structure.
		\item When $(m_1, m_2)\neq(4, 4k-5)$ (any positive integer $k>1$ ), each homogeneous isoparametric hypersurface admits an invariant complex structure.
	\end{enumerate}
\end{thm}

\begin{rem}
	\begin{enumerate}
		\item Homogeneous isoparametric hypersurfaces are classified by \cite{HL71} and \cite{TT72}. When $g=4$, the multiplicities $(m_1, m_2)=(1, k-2), (2, 2k-3), (4, 4k-5), (2, 2), (4, 5), (6, 9)$. In the case $(g, m_1, m_2)=(4, 2, 2)$, the isoparametric hypersurface is diffeomorphic to $SO(5)/T^2$ with Euler characteristic $\chi(SO(5)/T^2)=8$, and the number of invariant complex strucutres is equal to the order of the Weyl group $|W(SO(5))|=8$.  By Theorem B of \cite{Wan54}, any simply connected closed manifold with vanishing Euler characteristic admits either no homogeneous complex structure or non-countably many inequivalent ones. For the other cases with $g=4$ and $(m_1, m_2)\neq (2, 2)$, each isoparametric hypersurface has vanishing Euler characteristic. Therefore, according to part (ii) of Theorem \ref{homogeneous g=4}, each simply connected homogeneous isoparametric hypersurface with $g=4$ and $(m_1, m_2)\neq (2, 2), (4, 4k-5)$ (any positive integer $k>1$) admits non-countably many inequivalent homogeneous complex structures.
		In fact, this result can also be proven directly by carefully studying the invariant complex structures in our proof.\vspace{2mm}
		
		\item For the case $(g, m_1, m_2)=(4, 4, 4k-5)$, the isoparametric hypersurface $M$ is diffeomorphic to $S^4\times M_+$, it is unknown whether $M$ admits a complex structure (as it is not so symmetric).
	\end{enumerate}
\end{rem}

Finally, the special geometric structure of the complex structure associated with isoparametric theory will be discussed in Propositions 7 and 8 in Section \ref{Sec5}.
\vspace{2mm}

\section{\textbf{Almost Complex Structures and Isomorphic Principal Distributions}}

As for the existence of almost complex structures on isoparametric hypersurfaces  and focal submanifolds of OT-FKM type, we clarify
\vspace{3mm}

\noindent
\textbf{Proposition \ref{almost complex}}.
\emph{	In an isoparametric family with $g=4$ in the unit sphere, the following results hold:
\begin{enumerate}
	\item Every isoparametric hypersurface admits an almost complex structure.
	\item Every even-dimensional focal submanifold $M_+$ of OT-FKM type admits an almost complex structure.
	\item An even-dimensional focal submanifold $M_-$ of OT-FKM type admits an almost complex structure if and only if $M_-$ is diffeomorphic to a product of two odd-dimensional spheres.
\end{enumerate}}

\begin{proof}
For part (i), according to the classification, each isoparametric hypersurface with $g=4$ must be of OT-FKM type or homogeneous with $(m_1, m_2)=(2, 2), (4,5).$ By Theorem 1.6 of \cite{QTY23},
each isoparametric hypersurface of OT-FKM type admits an almost complex structure.
If $(m_1, m_2)=(2, 2), (4,5)$, then the isoparametric hypersurfaces are homogeneous, and diffeomorphic to $SO(5)/T^2$ and $U(5)/(SU(2)\times SU(2)\times U(1))$, respectively. Recall that a homogeneous space $G/H$ with $\mathrm{dim}~G/H=2n$ admits an invariant almost complex structure if and only if the isotropy representation can be factored through the standard inclusion of $U(n)$ in $SO(2n)$. In the proof of Theorem \ref{homogeneous g=4} for the $(SO(5)\times SO(5), \triangle SO(5))$ and $(SO(10), U(5))$ cases, the isotropy representations are well understood by using the restricted root decomposition. It follows that the isoparametric hypersurface with $(m_1, m_2)=(2, 2)$ or $(4,5)$ must admit an invariant almost complex structure.

Moreover, Theorem 1.6 of \cite{QTY23} also implies part (ii). Finally, part (iii) can be obtained by combining Theorem 1.6 of \cite{QTY23} and Corollary 1.3 of \cite{Sut65}. Furthermore, such $M_-$ are classified in Theorem 2.1 of \cite{QTY23}.
\end{proof}

From this point onward, we will focus on the existence of complex structures on isoparametric hypersurfaces with $g=4$ and focal submanifolds $M_+$ of OT-FKM type.

Although every isoparametric hypersurface with $g=4$ admits almost complex structures,  finding a special integrable one remains challenging, even if it exists. We aim to leverage the geometry of isoparametric theory for insight into this problem. To this end, we examine the properties of the principal distributions of isoparametric hypersurfaces $M$ of OT-FKM type.
Without loss of generality, assume $M := F^{-1}(0) \cap S^{2l-1}(1)$ for convenience.
The unit normal vector $\xi$ at $x \in M$ is given by:
\begin{equation}\label{xi}
\xi(x) = x - 2 \sum_{i=0}^m \langle P_i x, x \rangle P_i x :=x-2\eta(x).
\end{equation}
Clearly,  $X \in T_x M$ if and only if $\langle X, x \rangle = \langle X, \eta(x) \rangle = 0.$

The tangent space can be decomposed as the direct sum of the principal distributions:
$$T_x M = \mathcal{D}_1 \oplus  \mathcal{D}_2 \oplus  \mathcal{D}_3 \oplus  \mathcal{D}_4$$
corresponding to the principal curvatures
\begin{eqnarray*}
\lambda_1=	\text{cot} \left( \frac{\pi}{8}  \right)=\sqrt{2} + 1,&&\lambda_2=\text{cot} \left( \frac{\pi}{8} + \frac{\pi}{4} \right)= \sqrt{2} - 1\\
\lambda_3=\text{cot} \left( \frac{\pi}{8}+ \frac{\pi}{2} \right)=1 - \sqrt{2}, &&\lambda_4=\text{cot} \left( \frac{\pi}{8}+ \frac{3\pi}{4} \right)= - (\sqrt{2}+1).
\end{eqnarray*}
Clearly, $\xi(x)$ is the transport from $x\in M$ at distance $t=\frac{\pi}{2}$ along the normal geodesic of $M$. Therefore, we have the following result:
\begin{prop}
$\xi: M \to M$ is a diffeomorphism.
\end{prop}

Furthermore, we can prove the following:

\begin{lem}
At any $x\in M$, the map $P: \mathcal{D}_1\rightarrow \mathcal{D}_3$, defined by $X\mapsto Y=P(X)$ with $P=\sum_{i=0}^m \langle P_i x, x \rangle P_i$, is an isomorphism between the principal distributions $\mathcal{D}_1$ and $\mathcal{D}_3$.
\end{lem}

\begin{proof}
Note that the point at an oriented distance $t$ to $M$ along the
normal geodesic through $x$ can be defined as
$$\phi_{t}: M\rightarrow S^{n+1}(1),\quad x\mapsto \phi_{t}(x)=\cos t ~x+\sin t~\xi(x).$$
Clearly, when $g=4$ and $t=\frac{\pi}{2}$, $\phi_{\frac{\pi}{2}}: M\rightarrow M$  and $\phi_{\frac{\pi}{2}}(x)=\xi(x)$. Therefore, the tangent map is
$$\left(\phi_{\frac{\pi}{2}}\right)_*(X)=\xi_*(X)=\frac{\sin(\theta_i-\frac{\pi}{2})}{\sin\theta_i}X=-\cot\theta_i X=-\lambda_iX$$ for $X\in \mathcal{D}_i\subset T_xM$.

It suffices to show that for any $X \in \mathcal{D}_1$ (i.e., $\xi_{*} X = -\lambda_1 X$), the image $Y=P(X)$ satisfies $\xi_{*} Y = \frac{1}{\lambda_1} Y=-\lambda_3Y$, i.e., $Y \in \mathcal{D}_3$.

By a direct computation of the tangent map by (\ref{xi}), we obtain:
$$\xi_{*} X =X - 2 \sum \langle P_i x, x \rangle P_i X- 4 \sum \langle P_i x, X \rangle P_i x.$$
Rewriting, we define:
\begin{equation*}
	P(X) := PX=\sum_{i=0}^m \langle P_i x, x \rangle P_iX,\quad Q(X) := \sum_{i=0}^m \langle P_i x, X \rangle P_i x,
\end{equation*}
thus
$$\xi_{*} X:=X - 2 P(X)- 4 Q(X).$$
Moreover, it is satisfied that $P^2 = \frac{1}{2} I$ and $Q(X) = X^V$, $Q^2 = Q$, where $X^V$ is the projection of
$X$ onto $V = \text{Span} \{ P_0 x, \ldots, P_m x \}$.
Therefore, we have $P(X), Q(X) \in T_x M$ for all $X \in T_x M$.

Given $\xi_{*} X = -\lambda_1 X$, we have $-\lambda_1 X = X - 4 Q(X) - 2 P(X)$.
This implies that $Y=P(X) = \frac{\lambda_1 + 1}{2} X - 2 Q(X)$. Next, compute $\xi_{*}Y$ as follows:
\begin{eqnarray*}
	\xi_{*}Y&=&Y - 4 Q(Y) - 2 P(Y)\\
	&=&Y- 4 \left( \frac{\lambda_1 + 1}{2} Q(X) - 2 Q(X) \right) - 2 \left( \frac{1}{2} X \right) \\
	&=&P(X) - 2(\lambda_1 - 3) Q(X) - X.
\end{eqnarray*}
Therefore, $\xi_{*}Y = \frac{1}{\lambda_1} Y$ holds if and only if
$$(1 - \frac{1}{\lambda_1}) P(X) - 2 (\lambda_1 - 3) Q(X) - X = 0.$$
Rearranging, this becomes $(1 - \frac{1}{\lambda_1}) P(X) - 2 (\lambda_1 - 3) Q(X) = \frac{2}{\lambda_1 + 1} P(X) + \frac{4}{\lambda_1 + 1} Q(X)$, and further
$$\frac{\lambda_1^2 - 2\lambda _1- 1}{\lambda_1 (\lambda_1 + 1)} P(X) = 2 \frac{\lambda_1^2 - 2\lambda _1- 1}{\lambda_1 + 1} Q(X).$$
On the other hand,  since $\lambda_1 = \sqrt{2} + 1$, the identity $\lambda_1^2 - 2\lambda_1 - 1 = 0$ is satisfied. This completes the proof.
\end{proof}

Up to now, the following question remains unresolved in general:
\begin{ques}\label{isomorphic}
Are $\mathcal{D}_2$ and $\mathcal{D}_4$ isomorphic as vector bundles over $M$?
\end{ques}
\begin{rem}
According to P. 297 of Abresch \cite{Abr83}, $\mathcal{D}_2$ and $\mathcal{D}_4$ have the same total Stiefel-Whitney class.
\end{rem}
If the answer to Question \ref{isomorphic} is affirmative, then a natural follow-up question arises:
\begin{ques}\label{com str preserve}
Does there exist a complex structure $J$ such that $J\mathcal{D}_1=\mathcal{D}_3$ and $J\mathcal{D}_2=\mathcal{D}_4$?
\end{ques}

\section{\textbf{Complex Structures on the Isoparametric Hypersurfaces with $g=4, m=1$ }}\label{M1}
This section is dedicated to the proof of Theorem \ref{m=1}. We generalize Cartan's moving frame method, originally applied to isoparametric hypersurfaces with $(g, m_1, m_2)=(4, 1, 1)$, to the those with $(g, m_1, m_2)=(4, 1, l-2)$ (cf. \cite{Car40}).

When $m=1$, consider $\mathbb{C}^{l}=\mathbb{R}^{l}\oplus\mathbb{R}^{l}$ with the standard inner product $\langle \cdot, \cdot\rangle$.
According to \cite{FKM81}, up to isomorphism, we can express $P_0,\cdots, P_m$ as
\begin{equation*}\label{FKM}
P_0=\left(
\begin{matrix}
	I_l & 0 \\
	0 & -I_l \\
\end{matrix}\right),\;
P_1=\left(
\begin{matrix}
	~0 & I_l~\\
	~I_l & 0 ~\\
\end{matrix}\right),
P_{i+1}=\left(
\begin{matrix}
	0 & E_{i}\\
	-E_{i} & 0 \\
\end{matrix}\right),
~\text{for}~ 1 \leqslant i\leqslant m-1,
\end{equation*}
where $E_1, \cdots, E_{m-1}$ are representations of the generators of the Clifford algebra $\mathcal{C}_{m-1}$.
For $x=u+\sqrt{-1}v\in \mathbb{C}^{l}$, the corresponding polynomial is
$$F(x) = |x|^4 - 2\langle
P_{0}x,x\rangle^2-2\langle
P_{1}x,x\rangle^2,$$
which induces an isoparametric function $f=F|_{S^{2l-1}(1)}$ on $S^{2l-1}(1)$ and an isoparametric family in $S^{2l-1}$ with $(g, m_1, m_2)=(4, 1, l-2)$. For $0<t<\frac{\pi}{4}$, define an immersion
$$\varphi_t: S^1\times V_2(\mathbb{R}^l)\rightarrow \mathbb{R}^{2l}, \quad (e^{\sqrt{-1}\theta};u,v)\mapsto e^{\sqrt{-1}\theta}\left(\cos t~u+\sqrt{-1}\sin t~v\right),$$
where $V_2(\mathbb{R}^l)$ is the Stiefel manifold of orthonormal 2-frames in $\mathbb{R}^l$.
This immersion yields the isoparametric hypersurface $M_t=\varphi_t(S^1\times V_2(\mathbb{R}^l))=f^{-1}(-\cos 4t)$ in $S^{2l-1}(1)$.
For $x=u+\sqrt{-1}v:=(u, v)\in V_2(\mathbb{R}^l)$, choose a local orthonormal frame $e_1, e_2,\cdots, e_l$ of $\mathbb{R}^l$ with respect to the standard inner product $\langle \cdot, \cdot \rangle$, such that $e_1=u, e_2=v$. For $1\leq i, j\leq l$, let $\omega_{ij}:=\langle de_i, e_j\rangle$ and $\omega=(\omega_{ij})_{1\leq i, j\leq l}$. Then  $d\omega=\omega\wedge\omega$ and $\omega^T=-\omega.$ Consequently, the differential of $\varphi_t$ is
\begin{eqnarray*}
d\varphi_t&=&\frac{1}{\sqrt{2}}e^{\sqrt{-1}\theta}(-v+\sqrt{-1}u)\alpha+\frac{1}{\sqrt{2}}e^{\sqrt{-1}\theta}(-v-\sqrt{-1}u)\beta+\sum_{j=3}^{l}e^{\sqrt{-1}\theta}e_j\cos t~\omega_{1j}\nonumber\\
&&+\sum_{j=3}^{l}e^{\sqrt{-1}\theta}\sqrt{-1}e_j\sin t~\omega_{2j},
\end{eqnarray*}
where $\alpha=\frac{\sin t+\cos t}{\sqrt{2}}(d\theta-\omega_{12})$ and $\beta=\frac{\sin t-\cos t}{\sqrt{2}}(d\theta+\omega_{12})$. Define a unit normal vector field of the immersion $\varphi_t$ by
$\xi:=e^{\sqrt{-1}\theta}(-\sin t~u+\sqrt{-1}\cos t~v).$ Then
\begin{eqnarray*}
d\xi&=&e^{\sqrt{-1}\theta}\sqrt{-1}\left(-\sin t~u+\sqrt{-1}\cos t~v\right)d\theta+e^{\sqrt{-1}\theta}\left(-\sin t~v-\sqrt{-1}\cos t~u\right)\omega_{12}\nonumber\\
&&-\sum_{j=3}^{l}e^{\sqrt{-1}\theta}\sin t~e_j\omega_{1j}+\sum_{j=3}^{l}e^{\sqrt{-1}\theta}\sqrt{-1}\cos t~e_j\omega_{2j}.
\end{eqnarray*}
Moreover, the first and second fundamental forms of $\varphi_t$ are given by
\begin{eqnarray*}
I_t&=&\alpha^2+\beta^2+\sum_{j=3}^{l}\phi_j^2+\sum_{j=3}^{l}\psi_j^2,\nonumber\\
II_t&=&\cot (t+\frac{\pi}{4})\alpha^2+\cot (t+\frac{3\pi}{4})\beta^2+\sum_{j=3}^{l}\cot(t+\frac{\pi}{2})\phi_j^2+\sum_{j=3}^{l}\cot t~\psi_j^2,
\end{eqnarray*}
where $\phi_j=\cos t~\omega_{1j}$ and $\psi_j=\sin t~\omega_{2j}$.
For $1\leq j\leq 4$, define $\lambda_j:=\cot (t+j\frac{\pi}{4})$. The principal distributions of the principal curvatures corresponding to $\lambda_j$ are as follows:
\begin{eqnarray*}
{\mathcal{D}}_1&=&\mathrm{Span}\{\frac{1}{\sqrt{2}}e^{\sqrt{-1}\theta}(v-\sqrt{-1}u)\},\nonumber\\
{\mathcal{D}}_2&=&\mathrm{Span}\{e^{\sqrt{-1}\theta}e_j~|~3\leq j\leq l\},\nonumber\\
{\mathcal{D}}_3&=&\mathrm{Span}\{\frac{1}{\sqrt{2}}e^{\sqrt{-1}\theta}(v+\sqrt{-1}u)\},\nonumber\\
{\mathcal{D}}_4&=&\mathrm{Span}\{e^{\sqrt{-1}\theta}\sqrt{-1}e_j~|~3\leq j\leq l\}.
\end{eqnarray*}

Define $J_0:=P_0P_1$. It induces an almost complex structure on $M_t$ as follows.
Since $J_0x=\sqrt{-1}x$ and $J_0\xi=\sqrt{-1}\xi \in TM_t$, we have an orthogonal decomposition $$TM_t=\mathrm{Span}\{\sqrt{-1}x, \sqrt{-1}\xi\}\oplus \mathfrak{D},$$
Clearly, $\mathfrak{D}$ is $J_0$-invariant. According to \cite{TY22}, an almost complex structure $\widetilde{J}$ on $M_t$ can be defined
by
\begin{eqnarray}\label{overline J}
&&\widetilde{J}(\sqrt{-1}x)=\sqrt{-1}\xi, ~~\,\,~\widetilde{J}(\sqrt{-1}\xi)=-\sqrt{-1}x, \\
&&\qquad \widetilde{J}(X)=J_0X=\sqrt{-1}X,~~\text{for} ~~X\in \mathfrak{D}. \nonumber
\end{eqnarray}
Thus we have $\widetilde{J}(\frac{1}{\sqrt{2}}e^{\sqrt{-1}\theta}(v-\sqrt{-1}u))=\frac{1}{\sqrt{2}}e^{\sqrt{-1}\theta}(v+\sqrt{-1}u)$ and  $\widetilde{J}(e^{\sqrt{-1}\theta}e_j)=e^{\sqrt{-1}\theta}\sqrt{-1}e_j$, which implies that $\widetilde{J}{\mathcal{D}}_1={\mathcal{D}}_3$.
and $\widetilde{J}{\mathcal{D}}_2={\mathcal{D}}_4$.
However, Yi Zhou proved in his thesis that $\widetilde{J}$ is not integrable. For completeness, we give here a direct proof using Chern's characterization for the integrability of almost complex structures(see P. 15 of \cite{Ch95}).
\begin{prop}
The almost complex structure $\widetilde{J}$ in (\ref{overline J}) on $M_t$ is not integrable.
\end{prop}
\begin{proof}
By the definition of $\widetilde{J}$, we have
\begin{eqnarray*}
	&&\widetilde{J}(\frac{1}{\sqrt{2}}e^{\sqrt{-1}\theta}(v-\sqrt{-1}u))=\frac{1}{\sqrt{2}}e^{\sqrt{-1}\theta}(v+\sqrt{-1}u), \\
	&&\widetilde{J}(e^{\sqrt{-1}\theta}e_j)=e^{\sqrt{-1}\theta}\sqrt{-1}e_j,~3\leq j\leq l,
\end{eqnarray*}
and
$$\frac{1}{\sqrt{2}}e^{\sqrt{-1}\theta}(v-\sqrt{-1}u), ~~~\widetilde{J}(\frac{1}{\sqrt{2}}e^{\sqrt{-1}\theta}(v-\sqrt{-1}u)),~~~ e^{\sqrt{-1}\theta}e_j,~~~ \widetilde{J}(e^{\sqrt{-1}\theta}e_j) (3\leq j\leq l)$$
is a  unitary frame. 
Define $\theta:=-\alpha-\sqrt{-1}\beta$, and $\theta_j:=\cos t~\omega_{1j}+\sqrt{-1}\sin t~\omega_{2j}, 3\leq j\leq l$. Then
$$d\theta_3=\cos t~\omega_{1k}\wedge\omega_{k3}+\sqrt{-1}\sin t~ \omega_{2k}\wedge\omega_{k3}.$$
Since
$$\cos t~\omega_{12}\wedge\omega_{23}+\sqrt{-1}\sin t~\omega_{21}\wedge\omega_{13}\in \mathrm{Span}_{\mathbb{C}}\{\theta\wedge \theta_3, \overline{\theta}\wedge \theta_3, \theta\wedge \overline{\theta_3}, \overline{\theta}\wedge \overline{\theta_3}\},$$
and
$$\cos t~\omega_{12}\wedge\omega_{23}+\sqrt{-1}\sin t~\omega_{21}\wedge\omega_{13} \not\equiv 0~\mathrm{mod}~(\theta, \theta_3,\cdots, \theta_l),$$
it follows that $d\theta_3\not\equiv 0~\mathrm{mod}~(\theta, \theta_3,\cdots, \theta_l).$
According to \cite{Ch95}, the almost complex structure $\widetilde{J}$ on $M_t$ is not integrable.
\end{proof}

Next, we modify the definition of $\widetilde{J}$ to get a new almost complex structure $J$ and show its integrability in Theorem \ref{m=1}.
To be precise, $J$ is defined by
\begin{eqnarray}\label{ new J}
J(\frac{1}{\sqrt{2}}e^{\sqrt{-1}\theta}(v-\sqrt{-1}u))&=&-\cot (t+\frac{\pi}{4})\frac{1}{\sqrt{2}}e^{\sqrt{-1}\theta}(v+\sqrt{-1}u),\nonumber\\
J(\frac{1}{\sqrt{2}}e^{\sqrt{-1}\theta}(v+\sqrt{-1}u))&=&-\cot (t+\frac{3\pi}{4})\frac{1}{\sqrt{2}}e^{\sqrt{-1}\theta}(v-\sqrt{-1}u),\\
J(X)&=&-\cot (t+\frac{\pi}{2})\widetilde{J}(X)~\mathrm{for}~X\in \mathcal{D}_2,\nonumber\\
J(X)&=&\cot (t)\widetilde{J}(X)~\mathrm{for}~X\in \mathcal{D}_4.\nonumber
\end{eqnarray}

\noindent
\textbf{Theorem \ref{m=1}.}
\emph{The almost complex structure $J$ defined in (\ref{ new J}) on the isoparametric hypersurface $M_t$ in $S^{2l-1}$ with $(g, m_1, m_2)=(4, 1, l-2)$ is integrable. Moreover, $J{\mathcal{D}}_1={\mathcal{D}}_3$ and $J{\mathcal{D}}_2={\mathcal{D}}_4$.}
\vspace{2mm}

\begin{proof}
By the definition of $J$,
\begin{eqnarray*}
	&&\frac{\sin t+\cos t}{2}e^{\sqrt{-1}\theta}(v-\sqrt{-1}u),~\,\,~\cos t~e^{\sqrt{-1}\theta}e_j,\\
	&&J(\frac{\sin t+\cos t}{2}e^{\sqrt{-1}\theta}(v-\sqrt{-1}u))=\frac{\sin t-\cos t}{2}e^{\sqrt{-1}\theta}(v+\sqrt{-1}u), \\
	&&J(\cos t~e^{\sqrt{-1}\theta}e_j)
	=\sin t~e^{\sqrt{-1}\theta}\sqrt{-1}e_j,~3\leq j\leq l,
\end{eqnarray*}
is a unitary frame. Then define
$\theta:=-(d\theta-\omega_{12})-\sqrt{-1}(d\theta+\omega_{12})$, and $\theta_j:=\omega_{1j}+\sqrt{-1}\omega_{2j}, 3\leq j\leq l$.
It follows that
\begin{eqnarray*}
	d\theta &=&d\omega_{12}-\sqrt{-1}d\omega_{12},\nonumber\\
	&=&(1-\sqrt{-1})\sum_{k=3}^{l}\omega_{1k}\wedge \omega_{k2},\nonumber\\
	&=&(\sqrt{-1}-1)\sum_{k=3}^{l}\theta_k\wedge \omega_{2k},\nonumber\\
	&\equiv& 0~\mathrm{mod}~(\theta_3, \cdots, \theta_l).
\end{eqnarray*}
Moreover, for $3\leq j\leq l$,
\begin{eqnarray*}
	d\theta_j &=&d\omega_{1j}+\sqrt{-1}d\omega_{2j},\nonumber\\
	&=&\sum_{k=1}^{l}\omega_{1k}\wedge \omega_{kj}+\sqrt{-1}\sum_{k=1}^{l}\omega_{2k}\wedge \omega_{kj},\nonumber\\
	&=&-\sqrt{-1}\omega_{12}\wedge(\omega_{1j}+\sqrt{-1}\omega_{2j})+\sum_{k=3}^{l}(\omega_{1k}+\sqrt{-1}\omega_{2k})\wedge \omega_{kj},\nonumber\\
	&=&-\sqrt{-1}\omega_{12}\wedge\theta_j+\sum_{k=3}^{l}\theta_k\wedge \omega_{kj},\nonumber\\
	&\equiv&0~\mathrm{mod}~(\theta_3, \cdots, \theta_l).
\end{eqnarray*}
According \cite{Ch95}, $J$ is integrable and the proof is finished.
\end{proof}

\begin{rem}
Theorem \ref{m=1} provides an affirmative answer to Question \ref{com str preserve} in the case $g=4, m=1$ .
\end{rem}

More generally, for two arbitrary non-zero real numbers $\lambda$ and $\mu$, we can define an almost complex structure
$J_{\lambda, \mu}$ on $M_t$ by
\begin{eqnarray}\label{new J'}
J_{\lambda, \mu}(\frac{1}{\sqrt{2}}e^{\sqrt{-1}\theta}(v-\sqrt{-1}u))&=&-\lambda\cot (t+\frac{\pi}{4})\frac{1}{\sqrt{2}}e^{\sqrt{-1}\theta}(v+\sqrt{-1}u),\nonumber\\
J_{\lambda, \mu}(\frac{1}{\sqrt{2}}e^{\sqrt{-1}\theta}(v+\sqrt{-1}u))&=&-\frac{1}{\lambda}\cot (t+\frac{3\pi}{4})\frac{1}{\sqrt{2}}e^{\sqrt{-1}\theta}(v-\sqrt{-1}u),\\
J_{\lambda, \mu}(X)&=&-\mu\cot (t+\frac{\pi}{2})\widetilde{J}(X)\quad \mathrm{for}~X\in \mathcal{D}_2,\nonumber\\
J_{\lambda, \mu}(X)&=&\frac{1}{\mu}\cot (t)\widetilde{J}(X)\quad \mathrm{for}~X\in \mathcal{D}_4.\nonumber
\end{eqnarray}
Proving in a manner similar to Theorem \ref{m=1}, we can obtain
\begin{cor}
The almost complex structure
$J_{\lambda, \mu}$ on $M_t$ defined in (\ref{new J'}) is integrable if and only if $\mu^2=1$.
\end{cor}

\vspace{2mm}

\section{\textbf{Construction of Complex Structures on Focal Submanifolds}}\label{M_+ with m=2,4}

This section is dedicated to the proof of Theorem \ref{M+ complex}.
According to the expression of the polynomial $F$ in (\ref{FKM isop. poly.}), the focal submanifold $M_+=f^{-1}(1)$ of OT-FKM type is expressed as
\begin{eqnarray}
M_+
&=&\left\{x\in S^{2l-1}(1)~|~\langle P_0x, x\rangle=\langle P_1x, x\rangle=\cdots=\langle P_mx, x\rangle=0\right\}.	\label{M+}
\end{eqnarray}
In this section, we will construct complex structures on $M_+$ of OT-FKM type with $m=2$ in \ref{M+2}, \ref{M+2'} and with $m=4$ in the definite case in \ref{M+4}.

\subsection{\textbf{Complex structure on $M_+$ of OT-FKM type with $m=2$}}\label{M+2}

In this case, the focal submanifold $M_+$ of OT-FKM type can be expressed as
$$M_+=\{x\in\mathbb{R}^{2l}~|~|x|=1, \langle x, P_0x\rangle=\langle x, P_1x\rangle=\langle x, P_2x\rangle=0\}\subset \mathbb{R}^{2l}.$$
The normal space of $M_+$ in $\mathbb{R}^{2l}$ at $x\in M_+$ is spanned by $x, P_0x, P_1x, P_2x$, i.e.,
$$(T_xM_+)^{\perp}=\mathrm{Span}\{x, P_0x, P_1x, P_2x\}.$$
Using the property that $P_{i}P_{j}+P_{j}P_{i}=2\delta_{ij}I_{2l}$, we can verify that $P_0P_1x,P_0P_1P_2x, x, P_{\alpha}x$  $(\alpha=0,1,2)$ are perpendicular to each other. 
Equivalently speaking, $P_0P_1x, P_0P_1P_2x\in T_xM_+$, 
and $T_xM_+$ can be split into
$$T_xM_+=\textrm{Span}\{P_0P_1x\}\oplus\textrm{Span}\{P_0P_1P_2x\}\oplus E.$$
It is straightforward to verify that $P_0P_1X \in E$ for any $X\in E$. 
This observation allows us to construct a global endomorphism $J$ of $TM_+$ as follows:
\begin{equation}\label{J2}
\begin{split}
	&\hspace{3cm}J:\quad T_xM_+\longrightarrow T_xM_+\\
	& JP_0P_1x:=P_0P_1P_2x, ~~JP_0P_1P_2x:=-P_0P_1x, ~~JX:=P_0P_1X ~\textrm{for}~X\in E.
\end{split}
\end{equation}
Evidently, $J$ is an almost complex structure on $M_+$. Furthermore, $J$ is compatible with the induced metric $ds^2$ on $M_+$, making it almost Hermitian.

Next, we will prove the integrability of the almost complex structure $J$ in (\ref{J2}). From the definition (\ref{N}) of Nijenhuis tensor, one sees that it satisfies the following properties for any $X, Y\in\mathfrak{X}(M)$:
\begin{equation}\label{N'}
	N(JX, Y)=N(X, JY)=-JN(X, Y), \quad N(JX, JY)=-N(X, Y).
\end{equation}
Then we need only to verify $N(X, Y)=0$ and $N(P_0P_1x, X)=0$ for any $X, Y\in E$.

From definition (\ref{J2}), it follows that
\begin{equation*}
N(X, Y)=[P_0P_1X, P_0P_1Y]-[X, Y]-J\left([P_0P_1X, Y]+[X, P_0P_1Y]\right).
\end{equation*}
Denote the Levi-Civita connection in $\mathbb{R}^{2l}$ by $D$. It is clear that
\begin{eqnarray*}
[P_0P_1X, P_0P_1Y]-[X, Y]&=&D_{P_0P_1X}P_0P_1Y-D_{P_0P_1Y}P_0P_1X-D_XY+D_YX\\
&=&P_0P_1D_{P_0P_1X}Y-P_0P_1D_{P_0P_1Y}X-D_XY+D_YX.
\end{eqnarray*}
Denote $W:=[P_0P_1X, Y]+[X, P_0P_1Y]$. Since $X, P_0P_1X, Y, P_0P_1Y\in E$, we have $W\in T_xM_+$. More precisely,
\begin{eqnarray*}
W&=&D_{P_0P_1X}Y-D_{Y}P_0P_1X+D_XP_0P_1Y-D_{P_0P_1Y}X\\
&=&D_{P_0P_1X}Y-P_0P_1D_{Y}X+P_0P_1D_XY-D_{P_0P_1Y}X,
\end{eqnarray*}
and
\begin{eqnarray*}
\langle W, P_0P_1x\rangle
&=&\langle D_{P_0P_1X}Y+ P_0P_1D_XY- P_0P_1D_{Y}X- D_{P_0P_1Y}X, ~~P_0P_1x\rangle \\
&=&-\langle Y, P_0P_1D_{P_0P_1X}x\rangle + \langle D_XY, x\rangle -\langle D_{Y}X, x\rangle +\langle X, P_0P_1D_{P_0P_1Y}x\rangle\\
&=&\langle Y, X\rangle-\langle Y, X\rangle+\langle X, Y\rangle-\langle X, Y\rangle\\
&=&0,
\end{eqnarray*}
\begin{eqnarray*}
\langle W, P_0P_1P_2x\rangle
&=&\langle D_{P_0P_1X}Y+ P_0P_1D_XY- P_0P_1D_{Y}X- D_{P_0P_1Y}X, ~~P_0P_1P_2x\rangle \\
&=&-\langle Y, P_0P_1P_2(P_0P_1X)\rangle + \langle D_XY, P_2x\rangle -\langle D_{Y}X, P_2x\rangle +\langle X, P_0P_1P_2(P_0P_1Y)\rangle\\
&=&\langle Y, P_2X\rangle-\langle Y, P_2X\rangle+\langle X, P_2Y\rangle-\langle X, P_2Y\rangle\\
&&=0.
\end{eqnarray*}
Therefore, $W\in E$ and $JW=P_0P_1W$. Thus
\begin{eqnarray*}
N(X, Y)&=&[P_0P_1X, P_0P_1Y]-[X, Y]-P_0P_1W\\
&=&P_0P_1D_{P_0P_1X}Y-P_0P_1D_{P_0P_1Y}X-D_XY+D_YX\\
&&-P_0P_1(D_{P_0P_1X}Y-P_0P_1D_{Y}X+P_0P_1D_XY-D_{P_0P_1Y}X)\\
&=&0.
\end{eqnarray*}

As for $N(P_0P_1x, X)$, we have
\begin{equation*}
N(P_0P_1x, X)=[P_0P_1P_2x, P_0P_1X]-[P_0P_1x, X]-J\left([P_0P_1P_2x, X]+[P_0P_1x, P_0P_1X]\right).
\end{equation*}
Denote $W':=[P_0P_1P_2x, X]+[P_0P_1x, P_0P_1X]\in T_xM_+$. We see
\begin{eqnarray*}
W'&=& D_{P_0P_1P_2x}X-D_XP_0P_1P_2x+D_{P_0P_1x}P_0P_1X-D_{P_0P_1X}P_0P_1x\\
&=&D_{P_0P_1P_2x}X-P_0P_1P_2X+P_0P_1D_{P_0P_1x}X+X,
\end{eqnarray*}
and
\begin{eqnarray*}
\langle W', P_0P_1x\rangle
&=&\langle D_{P_0P_1P_2x}X-P_0P_1P_2X+P_0P_1D_{P_0P_1x}X+X, ~~P_0P_1x\rangle \\
&=&-\langle X, P_0P_1(P_0P_1P_2x)\rangle - \langle X, P_2x\rangle -\langle X, P_0P_1x\rangle +\langle X, P_0P_1x\rangle\\
&=&0,
\end{eqnarray*}
\begin{eqnarray*}
\langle W', P_0P_1P_2x\rangle
&=&\langle D_{P_0P_1P_2x}X-P_0P_1P_2X+P_0P_1D_{P_0P_1x}X+X, ~~P_0P_1P_2x\rangle \\
&=&\langle X, x\rangle-\langle X, x\rangle-\langle X, P_0P_1P_2x\rangle+\langle X,  P_0P_1P_2x\rangle\\
&&=0.
\end{eqnarray*}
Therefore, $W'\in E$ and $JW_0'=P_0P_1W'$. Thus
\begin{eqnarray*}
N(P_0P_1x, X)&=&[P_0P_1P_2x, P_0P_1X]-[P_0P_1x, X]-P_0P_1W'\\
&=&P_0P_1D_{P_0P_1P_2x}X+P_2X-D_{P_0P_1x}X+P_0P_1X\\
&&-P_0P_1(D_{P_0P_1P_2x}X-P_0P_1P_2X+P_0P_1D_{P_0P_1x}X+X)\\
&=&0.
\end{eqnarray*}
Consequently, when $m=2$, the focal submanifold $M_+$ of OT-FKM type is a complex manifold.

\subsection{\textbf{An alternative complex structure on $M_+$ of OT-FKM type with $m=2$}}\label{M+2'}
In this subsection, we will show that $M_+$ of OT-FKM type with $m=2$ is diffeomorphic to a complex hypersurface of the modified Calabi-Eckmann manifold $S^{2k-1}\times S^{2k-1}$.

In fact, from the definition of the focal submanifold $M_+$ of OT-FKM type, we see that
\begin{eqnarray*}
M_+&\cong& N:=\{(x, y)\in S^{2k-1}\times S^{2k-1}~|~\langle x, y\rangle=\langle x, iy\rangle=0\}\\
&\subset& \widetilde{N}:=S^{2k-1}\times S^{2k-1}.
\end{eqnarray*}
To differentiate the positions of the two spheres, we add subscripts to them such that $\widetilde{N}=S_1^{2k-1}\times S_2^{2k-1}$. Embed $S^{2k-1}$ into $\mathbb{C}^k$, so we can split the tangent space $T_{(x, y)} S_1^{2k-1}\times S_2^{2k-1}$ at $(x, y)\in S_1^{2k-1}\times S_2^{2k-1}$ into $T_{(x, y)} S_1^{2k-1}\times S_2^{2k-1}=T_xS_1^{2k-1}\times T_yS_2^{2k-1}$ with $T_xS_1^{2k-1}=\textrm{Span}\{ix\}\oplus U$ and
$T_yS_2^{2k-1}=\textrm{Span}\{iy\}\oplus V$,
where we have used the notations
$ix:=(ix, 0), iy:=(0, iy)$ for simplicity. Denote $u:=(u, 0)\in U, v:=(0, v)\in V$,
we can define an endomorphism $J$ of $T\widetilde{N}$ as follows
\begin{equation}\label{J2'}
\begin{split}
	&\hspace{1cm}J:\quad T_{(x, y)}\widetilde{N}\longrightarrow T_{(x, y)}\widetilde{N}\\
	&J(ix):=iy, \hspace{2cm} J(iy):=-ix, \\
	&J(u):=iu, ~\textrm{for}~u\in U\quad J(v):=-iv ~\textrm{for}~v\in V.
\end{split}
\end{equation}
It's clear that $J^2=-Id$, thus $J$ is an almost complex structure on $\widetilde{N}=S^{2k-1}\times S^{2k-1}$.

\begin{lem}\label{lem4.1}
The almost complex structure $J$ of $\widetilde{N}$ in (\ref{J2'}) is integrable.
\end{lem}

\begin{proof}
By the Property (\ref{N}) of the Nijenhuis tensor, we need only to calculate $N(ix, u)$, $N(u_1, u_2)$, $N(ix, v)$, $N(v_1, v_2)$ and $N(u, v)$ for any $u, u_1, u_2\in U$ and $v, v_1, v_2\in V$.

\noindent
(1) For any $u\in U$,
\begin{eqnarray*}
	N(ix, u)&=&[iy, iu]-[ix, u]-J[iy, u]-J[ix, iu]\\
	&=&-[ix, u]-J[ix, iu]\\
	&=&-[ix, u]-i[ix, iu]\\
	&=&-D_{ix}u+D_{u}ix-iD_{ix}iu+iD_{iu}ix\\
	&=&0
\end{eqnarray*}
where the third equlity follows from the facts that $[ix, iu]\in T_xS_1^{2k-1}$ and
$$\langle ix, D_{ix}iu-D_{iu}ix\rangle =-\langle D_{ix}ix, iu\rangle -\frac{1}{2}iu\langle ix, ix\rangle=\langle x, iu\rangle=0,$$
since the Euclidean connection $D$ is used, which is K\"ahler.

\noindent
(2) For any $u_1, u_2\in U$,
\begin{eqnarray*}
	N(u_1, u_2)&=&[iu_1, iu_2]-[u_1, u_2]-J([iu_1, u_2]+[u_1, iu_2])\\
	&=&D_{iu_1}iu_2-D_{iu_2}iu_1-D_{u_1}u_2+D_{u_2}u_1\\
	&&-i(D_{iu_1}u_2-D_{u_2}iu_1+D_{u_1}iu_2-D_{iu_2}u_1)\\
	&=&0,
\end{eqnarray*}
where the second equality follows from the facts that $[iu_1, u_2]+[u_1, iu_2]\in T_xS_1^{2k-1}$ and
\begin{eqnarray*}
	&&\langle  ix, ~~D_{iu_1}u_2-D_{u_2}iu_1+D_{u_1}iu_2-D_{iu_2}u_1\rangle\\
	&=&\langle -D_{iu_1}ix, u_2\rangle +\langle D_{u_2}ix, iu_1\rangle-\langle D_{u_1}ix, iu_2\rangle +\langle D_{iu_2}ix, u_1\rangle\\
	&=&\langle u_1, u_2\rangle +\langle iu_2, iu_1\rangle-\langle iu_1, iu_2\rangle -\langle u_2, u_1\rangle\\
	&=& 0.
\end{eqnarray*}

\noindent
(3) For any $v\in V$,
\begin{eqnarray*}
	N(ix, v)&=&[iy, -iv]-[ix, v]-J[iy, v]-J[ix, -iv]\\
	&=&-D_{iy}iv+D_{iv}iy-J(D_{iy}v-D_viy)\\
	&=&-D_{iy}iv-v+i(D_{iy}v-iv)\\
	&=&0,
\end{eqnarray*}
where the third equality follows from the facts that $D_{iy}v-D_viy\in T_yS_2^{2k-1}$ and
$$\langle  iy, ~~D_{iy}v-D_viy\rangle=\langle -D_{iy}iy, v\rangle -\frac{1}{2}v\langle iy, iy\rangle=0.$$

\noindent
(4) For any $v_1, v_2\in V$,
\begin{eqnarray*}
	N(v_1, v_2)&=&[-iv_1, -iv_2]-[v_1, v_2]-J[-iv_1, v_2]-J[v_1, -iv_2]\\
	&=&D_{iv_1}iv_2-D_{iv_2}iv_1-D_{v_1}v_2+D_{v_2}v_1+J(D_{iv_1}v_2-D_{v_2}iv_1+D_{v_1}iv_2-D_{iv_2}v_1)\\
	&=&D_{iv_1}iv_2-D_{iv_2}iv_1-D_{v_1}v_2+D_{v_2}v_1-i(D_{iv_1}v_2-D_{v_2}iv_1+D_{v_1}iv_2-D_{iv_2}v_1)\\
	&=&0,
\end{eqnarray*}
where the third equality follows from the facts that $[-iv_1, v_2]+[v_1, -iv_2]\in T_yS_2^{2k-1}$ and
\begin{eqnarray*}
	&&\langle  iy, ~~D_{iv_1}v_2-D_{v_2}iv_1+D_{v_1}iv_2-D_{iv_2}v_1\rangle\\
	&=&\langle -D_{iv_1}iy, v_2\rangle +\langle D_{v_2}iy, iv_1\rangle-\langle D_{v_1}iy, iv_2\rangle +\langle D_{iv_2}iy, v_1\rangle\\
	&=&\langle v_1, v_2\rangle +\langle iv_2, iv_1\rangle-\langle iv_1, iv_2\rangle -\langle v_2, v_1\rangle\\
	&=& 0.
\end{eqnarray*}

\noindent
(5) For any $u\in U$, $v\in V$, it follws directly from the property of Lie bracket that
$$	N(u, v)=[iu, -iv]-[u, v]-J[iu, v]-J[u, -iv]=0.$$
\end{proof}

Moreover, at a point $(x, y)\in N$, we have $y\in U$ since $\langle y, x\rangle=\langle y, ix\rangle=0$, similarly, $x\in V$. Hence there exist two normal vector fields, say $\xi_1$, $\xi_2$,
$$\xi_1(x, y):=(y, x),\quad \xi_2(x, y):=(iy, -ix).$$
By Lemma \ref{lem4.1}, we see that $J\xi_1=\xi_2$, $J\xi_2=-\xi_1$. Therefore, $JX\in T_{(x, y)}N$ for any $X\in T_{(x, y)}N$. Consequently, $N$ is a complex hypersurface of the Calabi-Eckmann manifold $S^{2k-1}\times S^{2k-1}$.

\subsection{\textbf{Complex structure on $M_+$ of OT-FKM type with $m=4$ in the definite case}}\label{M+4}

As mentioned in the introduction, the definite case can be characterized by $P_0\cdots P_4=\pm I_{2l}$. The normal space of $M_+$ in $\mathbb{R}^{2l}$ at $x\in M_+$ is spanned by $x, P_0x,\cdots, P_4x$, i.e.,
$$(T_xM_+)^{\perp}=\mathrm{Span}\{x, P_0x, \cdots, P_4x\}.$$
Using the properties that $P_{i}P_{j}+P_{j}P_{i}=2\delta_{ij}I_{2l}$ and $P_0\cdots P_4=\pm I_{2l}$, we can verify directly that $P_0P_1x, P_2P_3x, P_2P_4x$ and $P_3P_4x$ are perpendicular to each other and to $(T_xM_+)^{\perp}$. Therefore, we can split $T_xM_+$ into
$$T_xM_+=\textrm{Span}\{P_0P_1x, P_2P_3x, P_2P_4x, P_3P_4x\}\oplus V.$$
For $X\in V$, it is easily seen that $P_0P_1X$ is perpendicular to $x, P_{\alpha}x ~ (\alpha=0,\cdots, 4)$ and $P_0P_1x, P_2P_3x, P_2P_4x, P_3P_4x$, implying that $P_0P_1X\in V$. This enables us to construct a global
endomorphism $J$ of $TM_+$ as follows:
\begin{equation}\label{J3}
\begin{split}
	&\hspace{1cm}J:\quad T_xM_+\longrightarrow T_xM_+\\
	& JP_0P_1x:=P_3P_4x, \quad JP_3P_4x:=-P_0P_1x,\\
	&  JP_2P_3x:=P_2P_4x, \quad JP_2P_4x:=-P_2P_3x,\\
	&\qquad JX:=P_0P_1X ~\textrm{for}~X\in V.
\end{split}
\end{equation}
It is obvious that $J$ is an almost complex structure on $M_+$ which is compatible with the induced metric $ds^2$, thus almost Hermitian.

Next, we show the integrability of $J$.

\noindent
(1) By definition,
\begin{eqnarray*}
N(P_0P_1x, P_2P_3x)&=&[P_3P_4x, P_2P_4x]-[P_0P_1x, P_2P_3x]-J[P_3P_4x, P_2P_3x]-J[P_0P_1x, P_2P_4x]\\
&=& D_{P_3P_4x}P_2P_4x-D_{P_2P_4x}P_3P_4x-D_{P_0P_1x}P_2P_3x+D_{P_2P_3x}P_0P_1x\\
&&-J(D_{P_3P_4x}P_2P_3x-D_{P_2P_3x}P_3P_4x+D_{P_0P_1x}P_2P_4x-D_{P_2P_4x}P_0P_1x)\\
&=&-2P_2P_3x -2JP_2P_4x\\
&=& 0.	
\end{eqnarray*}

\noindent
(2) For any $X\in V$,
\begin{eqnarray*}
N(P_0P_1x, X)&=&[P_3P_4x, P_0P_1X]-[P_0P_1x, X]-J[P_3P_4x, X]-J[P_0P_1x, P_0P_1X]\\
&=& P_0P_1D_{P_3P_4x}X-P_3P_4P_0P_1X-D_{P_0P_1x}X+P_0P_1X \\
&& -J(D_{P_3P_4x} X-P_3P_4X+D_{P_0P_1x}P_0P_1X+X)   \\
&=& P_0P_1D_{P_3P_4x}X-P_3P_4P_0P_1X-D_{P_0P_1x}X \\
&&-J(D_{P_3P_4x} X-P_3P_4X+D_{P_0P_1x}P_0P_1X )\\
\end{eqnarray*}

It is easily seen that
$D_{P_0P_1x}P_0P_1X$ is perpendicular to $x, P_{\alpha}x ~ (\alpha=0,\cdots, 4)$ and $P_0P_1x$, $P_2P_3x$, $P_2P_4x, P_3P_4x$, implying that $D_{P_0P_1x}P_0P_1X\in V$, and thus
$JD_{P_0P_1x}P_0P_1X=P_0P_1D_{P_0P_1x}P_0P_1X=-D_{P_0P_1x}X$.

As for $D_{P_3P_4x} X-P_3P_4X$, it is easily checked that it is perpendicular to $x, P_{\alpha}x ~ (\alpha=2, 3, 4)$ and $P_0P_1x$, $P_2P_3x$, $P_2P_4x, P_3P_4x$. Moreover, $$\langle D_{P_3P_4x} X-P_3P_4X, P_0x\rangle=-\langle X, P_0P_3P_4x\rangle-\langle X, P_4P_3P_0x\rangle =0,$$ similarly, $\langle D_{P_3P_4x} X-P_3P_4X, P_1x\rangle=0$. Therefore, $D_{P_3P_4x} X-P_3P_4X\in V$ and $$J(D_{P_3P_4x} X-P_3P_4X)=P_0P_1(D_{P_3P_4x} X-P_3P_4X)=P_0P_1D_{P_3P_4x} X-P_0P_1P_3P_4X,$$ which implies that $N(P_0P_1x, X)=0$.
\vspace{2mm}

\noindent
(3) For any $X\in V$,
\begin{eqnarray*}
N(P_2P_3x, X)&=&[P_2P_4x, P_0P_1X]-[P_2P_3x, X]-J[P_2P_4x, X]-J[P_2P_3x, P_0P_1X]\\
&=& P_0P_1D_{P_2P_4x}X-P_2P_4P_0P_1X-D_{P_2P_3x}X+P_2P_3X \\
&& -J(D_{P_2P_4x} X-P_2P_4X+P_0P_1D_{P_2P_3x}X-P_2P_3P_0P_1X)   \\
&=&0,
\end{eqnarray*}
where the last equality follows from the fact that $D_{P_2P_4x} X-P_2P_4X+P_0P_1D_{P_2P_3x}X$ $-P_2P_3P_0P_1X\in V$.
\vspace{2mm}

\noindent
(4) For any $X, Y\in V$,
\begin{eqnarray*}
N(X, Y)&=&[P_0P_1X, P_0P_1Y]-[X, Y]-J[P_0P_1X, Y]-J[X, P_0P_1Y]\\
&=& P_0P_1D_{P_0P_1X}Y-P_0P_1D_{P_0P_1Y}X-D_{X}Y+D_{Y}X \\
&& -J(D_{P_0P_1X} Y-P_0P_1D_{Y}X+P_0P_1D_{X}Y-D_{P_0P_1Y} X)   \\
&=&0,
\end{eqnarray*}
where the last equality follows from the fact that $D_{P_0P_1X} Y-P_0P_1D_{Y}X+P_0P_1D_{X}Y-D_{P_0P_1Y} X\in V$ which can be checked easily.

\section{\textbf{Homogeneous Almost Complex and Complex Structure}}\label{Sec5}

\subsection{Homogeneous almost complex structure.} 
Let $G$ be a compact Lie group and $H$ a closed subgroup. Then the quotient space $G/H$ is a smooth manifold with the natural transitive action by $G$. More precisely, given any $g\in G$, the action of $g$ on $G/H$ is defined by $\tau(g): G/H\rightarrow G/H, pH\mapsto gpH$, for any $p\in G$.

Choose a bi-invariant metric $Q$ on $G$. Let $\mathfrak{g}$ and $\mathfrak{h}$ be the Lie algebras of $G$ and $H$, respectively. Define $\mathfrak{m}:=\mathfrak{h}^{\bot}$. The isotropy representation of $H$ on $\mathfrak{m}$ is defined by $Ad:H\times \mathfrak{m}\rightarrow \mathfrak{m}, (h, v)\mapsto Ad(h)(v)=dL_h\circ dR_{h^{-1}}(v)$.

\begin{defn}
An almost complex structure $J$ on $G/H$ is called \textbf{$G$-invariant} if for any $g\in G$ and $v\in T_H(G/H)$, $J(d\tau_g(v))=d\tau_g(J(v))$.
\end{defn}

The basic principle for studying invariant structures of $G/H$ is that $G$-invariant geometric objects on $G/H$ correspond one-to-one with $Ad_{H}$-invariant algebraic objects in $\mathfrak{m}$.
In particular, a $G$-invariant almost complex structure $J$ corresponds one-to-one with an $Ad_H$-invariant linear transformation $I: \mathfrak{m}\rightarrow \mathfrak{m}$ with $I^2=-\text{id}$.

According to Koszul \cite{Kos55}, the $G$-invariant almost complex structure $J$ on $G/H$ is integrable if and only if the corresponding $I: \mathfrak{m}\rightarrow \mathfrak{m}$ satisfies:
for any $x, y\in \mathfrak{m}$, $$[Ix, Iy]-I[Ix, y]_{\mathfrak{m}}-I[x, Iy]_{\mathfrak{m}}-[x, y]\equiv 0~\mathrm{mod}~\mathfrak{h}.$$

\subsection{$M_+$ of OT-FKM with $m=8$.}
In this case, the isoparametric hypersurface of OT-FKM type is not extrinsically homogeneous(cf. \cite{FKM81}).
\vspace{2mm}

\noindent
\textbf{Proposition \ref{m=8 focal}}.
\emph{The focal submanifold $M_+$ with $(g, m_1, m_2)=(4, 8, 7)$ of OT-FKM type in the definite case admits a transitive action by $Spin(9)$ and is diffeomorphic to $Spin(9)/G_2$. Moreover, $M_+$ does not admit any $Spin(9)$-invariant almost complex structure.}

\begin{proof}
Let $Q=-B$ be a bi-invariant metric on $Spin(9)$, where $B$ is the Killing form of $Spin(9)$, and $\mathfrak{spin}(9)$ and $\mathfrak{g}_2$ be the Lie algebras of $Spin(9)$ and $G_2$, respectively. Then the isotropy representation of $G_2$ on $\mathfrak{m}:=\mathfrak{g}_2^{\bot}$ can be decomposed as $\mathfrak{m}=\mathfrak{m}_1\bigoplus \mathfrak{m}_2\bigoplus \mathfrak{m}_3\bigoplus \underline{\mathbf{1}},$ where $\underline{\mathbf{1}}$ is a $1$-dimensional trivial representation. Moreover, $\mathrm{dim}~\mathfrak{m}_1=\mathrm{dim}~\mathfrak{m}_2=\mathrm{dim}~\mathfrak{m}_3=7$, and for $1\leq i\leq 3$, each isotropy representation of $G_2$ on $\mathfrak{m}_i$ is irreducible. It follows from Schur's Lemma that $M_+$ does not admit any $Spin(9)$-invariant almost complex structure.

Now, the proof is finished.
\end{proof}

\subsection{Homogeneous complex structures on isoparametric hypersurfaces}

\subsubsection{Classification of homogeneous hypersurfaces in unit spheres}
In this section, we will focus on homogeneous isoparametric hypersurfaces with $g=4$ distinct principal curvatures. According to Hsiang-Lawson \cite{HL71} and Takagi-Takahashi \cite{TT72}, homogeneous isoparametric hypersurfaces in unit spheres must be principal orbits of isotropy representations of symmetric spaces of rank $2$ (see also \cite{TXY12}). For $g=4$, each homogeneous hypersurface $M$ is a principal orbit of the isotropy representation of \\
$$SO(5), ~\,~SO(10)/U(5),~\,~ E_6/(U(1)\cdot Spin(10)),$$
$$G^+_2(\mathbb{R}^{k+2})=SO(k+2)/(SO(2)\times SO(k)) (k\geq 3),$$
$$G_2(\mathbb{C}^{k+2})=SU(k+2)/S(U(2)\times U(k)) (k\geq 3),$$
$$G_2(\mathbb{H}^{k+2})=Sp(k+2)/(Sp(2)\times Sp(k)) (k\geq 2).$$

\subsubsection{Restricted root decompositions and geometry of homogeneous hypersurfaces}
To study the invariant geometric structures of orbits of isotropy representations, we will introduce the restricted root decomposition of symmetric spaces. Let $G/K$ be an irreducible Riemannian symmetric space of compact type, where $G$ is a connected compact semisimple Lie group and $K$ is its closed subgroup. Let $\mathfrak{g}$ and $\mathfrak{k}$ denote Lie algebras of $G$ and $K$, respectively. Then $G/K$ gives rise to an involutive automorphism $\theta$ of $\mathfrak{g}$ such that
$\mathfrak{k}=\{X\in \mathfrak{g}~|~\theta(X)=X\}$.
Denote by $\mathfrak{p}$ the eigenspace of $\theta$ corresponding to the eigenvalue $-1$, i.e., $\mathfrak{p}=\{X\in \mathfrak{g}~|~\theta(X)=-X\}$.
Clearly, we have
$$[\mathfrak{k}, \mathfrak{p}]\subset \mathfrak{p},\quad [\mathfrak{p}, \mathfrak{p}]\subset \mathfrak{k}, ~ \mathrm{and}~ \mathfrak{g}=\mathfrak{k}\oplus\mathfrak{p}.$$

Choose a maximal abelian subspace $\mathfrak{a}$ of $\mathfrak{p}$. The dimension of $\mathfrak{a}$ is called the rank of $G/K$. For a linear form $\alpha$ on $\mathfrak{a}$, define
\begin{eqnarray*}
\mathfrak{k}_{\alpha}&:=&\{X\in \mathfrak{k}~|~(ad H)^2(X)=-\alpha(H)^2X, ~~~\forall H\in \mathfrak{a}\},\nonumber\\
\mathfrak{p}_{\alpha}&:=&\{X\in \mathfrak{p}~|~(ad H)^2(X)=-\alpha(H)^2X, ~~~\forall H\in \mathfrak{a}\}.
\end{eqnarray*}
By definition, $\mathfrak{k}_{\alpha}=\mathfrak{k}_{-\alpha}, \mathfrak{p}_{\alpha}=\mathfrak{p}_{-\alpha}, \mathfrak{p}_{0}=\mathfrak{a}$ and $\mathfrak{k}_0$ is the centralizer of $\mathfrak{a}$ in $\mathfrak{k}$.
For each linear form $\alpha$, the dimension of $\mathfrak{p}_{\alpha}$ is called the multiplicity of $\alpha$, denoted by $m(\alpha)$. A restricted root of $\mathfrak{g}$ with respect to $\mathfrak{a}$ is a linear form $\alpha$ on $\mathfrak{a}$ such that $m(\alpha)\neq 0$. Fix a suitable ordering in $\mathfrak{a}^*$, the dual space of $\mathfrak{a}$, and denote by $\triangle$ the set of positive roots. Since $G$ is compact, we can define an inner product $\langle , \rangle $ on $\mathfrak{g}$ which is $Ad_G$-invariant. Then we have the following orthogonal decomposition:
\begin{eqnarray*}
\mathfrak{k}&=&\mathfrak{k}_0\bigoplus_{\alpha \in \triangle}\mathfrak{k}_{\alpha},\nonumber\\
\mathfrak{p}&=&\mathfrak{a}\bigoplus_{\alpha \in \triangle}\mathfrak{p}_{\alpha}.
\end{eqnarray*}
Moreover, for $\alpha\in\triangle$,
$[\mathfrak{a}, \mathfrak{k}_{\alpha}]= \mathfrak{p}_{\alpha},$ and $[\mathfrak{a}, \mathfrak{p}_{\alpha}]= \mathfrak{k}_{\alpha}.$
According to \cite{Hel65}, if $\alpha, \beta\in \triangle\cup\{0\},$ then
$$[\mathfrak{k}_{\alpha}, \mathfrak{k}_{\beta}]\subset \mathfrak{k}_{\alpha+\beta}+\mathfrak{k}_{\alpha-\beta},$$
$$[\mathfrak{k}_{\alpha}, \mathfrak{p}_{\beta}]\subset \mathfrak{p}_{\alpha+\beta}+\mathfrak{p}_{\alpha-\beta},$$
$$[\mathfrak{p}_{\alpha}, \mathfrak{p}_{\beta}]\subset \mathfrak{k}_{\alpha+\beta}+\mathfrak{k}_{\alpha-\beta}.$$
If $\alpha+\beta\in \triangle\cup\{0\}$ or $\alpha-\beta\in \triangle\cup\{0\}$, then $[\mathfrak{k}_{\alpha}, \mathfrak{p}_{\beta}]\neq 0.$
Let $S$ be the unit sphere in $\mathfrak{p}$. The isotropy representation of $K$ acts isometrically on $\mathfrak{p}$ by the adjoint representation, and the unit sphere $S\subset \mathfrak{p}$ is $K$-invariant.   For each $X$ in $\mathfrak{p}$, there is an element $k\in K$ such that $Ad(k)X\in \mathfrak{a}$. It follows that each $Ad(K)$-orbit in $\mathfrak{p}$ meets $\mathfrak{a}$.
For $H\in \mathfrak{a}\cap S$, define $H^{\bot}:=\{X\in \mathfrak{a}~|~\langle X, H \rangle=0\}$
and $\triangle_H:=\{\alpha\in \triangle~|~\alpha(H)\neq 0\}.$
Let $N$ be an orbit in $S$ under $K$ through $H$, i.e., $N=Ad(K)H$. Then the tangent space of $N$ at $H$ is given by $$T_{H}N=[H, \mathfrak{k}]=\bigoplus_{\alpha \in \triangle_H}\mathfrak{p}_{\alpha}.$$
It follows that the normal space of $N$ in $S$ at $H$ is determined by
$$T^{\bot}_{H}N=H^{\bot}\bigoplus_{\alpha \in \triangle-\triangle_H}\mathfrak{p}_{\alpha}.$$
Define $\triangle^*:=\{\alpha\in \triangle~|~\frac{1}{2}\alpha \not\in \triangle\}$. One has
\begin{prop}(\cite{TT72})
Assume dim $\mathfrak{a}=2$ and $N$ is an orbit of highest dimension(i.e., $\triangle=\triangle_H$). With the unit normal vector $B$ of $N$ in $S$ at $H$, $\{H, B\}$ is an orthogonal basis of $\mathfrak{a}$. The principal curvatures of $N$ with respect to $B$ are given by $-\frac{\alpha(B)}{\alpha(H)}$, where $\alpha \in \triangle^*$. Moreover, the principal distribution of $-\frac{\alpha(B)}{\alpha(H)}$ is $\mathfrak{p}_{\alpha}\oplus \mathfrak{p}_{2\alpha}$ and its multiplicity is $m(\alpha)+m(2\alpha)$.
\end{prop}

\subsubsection{Existence and non-existence of homogeneous complex structures}
\hspace{6mm}

\vspace{2mm}
\noindent
\textbf{Theorem \ref{homogeneous g=4}}
\emph{	 For homogeneous isoparametric hypersurfaces with $g=4$,
\begin{enumerate}
\item When $(m_1, m_2)=(4, 4k-5)$ ($k\geq 2$), each homogeneous isoparametric hypersurface $M^{8k-2}$ corresponding to the symmetric pair $(Sp(k+2), Sp(2)\times Sp(k))$
admits no $Sp(2)\times Sp(k)$-invariant almost complex structure.
\item When $(m_1, m_2)\neq(4, 4k-5)$ (any positive integer $k>1$ ), each homogeneous isoparametric hypersurface admits an invariant complex structure.
\end{enumerate}}

\begin{proof}
For the proof of part (i), the symmetric pair is $(G, K)=(Sp(k+2), Sp(2)\times Sp(k))$. By the restricted root space decomposition, we have
\begin{eqnarray*}
\mathfrak{k}&=&\mathfrak{k}_0\oplus \mathfrak{k}_{\alpha_1} \oplus \mathfrak{k}_{2\alpha_1} \oplus \mathfrak{k}_{\alpha_2} \oplus \mathfrak{k}_{2\alpha_2} \oplus\mathfrak{k}_{\alpha_1+\alpha_2} \oplus \mathfrak{k}_{\alpha_1-\alpha_2},\nonumber\\
\mathfrak{p}&=&\mathfrak{a}\oplus \mathfrak{p}_{\alpha_1} \oplus \mathfrak{p}_{2\alpha_1} \oplus \mathfrak{p}_{\alpha_2} \oplus \mathfrak{p}_{2\alpha_2}\oplus \mathfrak{p}_{\alpha_1+\alpha_2} \oplus \mathfrak{p}_{\alpha_1-\alpha_2}.
\end{eqnarray*}
The connected Lie subgroup $K_0$ corresponding to $\mathfrak{k}_0$ is $\Delta(Sp(1)\times Sp(1))\times Sp(k-2)$. Thus $M^{8k-2}\cong K/K_0$. Moreover, $\dim \mathfrak{k}_{\alpha_i}=\dim \mathfrak{p}_{\alpha_i}=4k-8$, $\dim \mathfrak{k}_{2\alpha_i}=\dim \mathfrak{p}_{2\alpha_i}=3$, and $\dim \mathfrak{k}_{\alpha_1\pm\alpha_2}=\dim \mathfrak{p}_{\alpha_1\pm \alpha_2}=4$.

Fix a bi-invariant metric $Q$ on $K$. Let $\mathfrak{m}=\mathfrak{k}_0^{\bot}\subset \mathfrak{k}$. Then for $K/K_0$, the isotropy representation of $K_0$ on $\mathfrak{m}$ has the irreducible decomposition $$\mathfrak{m}=\mathfrak{k}_{\alpha_1} \oplus \mathfrak{k}_{2\alpha_1} \oplus \mathfrak{k}_{\alpha_2} \oplus \mathfrak{k}_{2\alpha_2} \oplus\mathfrak{k}_{\alpha_1+\alpha_2} \oplus \mathfrak{k}_{\alpha_1-\alpha_2}.$$ Moreover, as representations of $K_0$, $\mathfrak{k}_{\alpha_1}\not\cong \mathfrak{k}_{\alpha_2}$, $\mathfrak{k}_{2\alpha_1}\not\cong \mathfrak{k}_{2\alpha_2}$, and $\mathfrak{k}_{\alpha_1+\alpha_2}\cong \mathfrak{k}_{\alpha_1-\alpha_2}$. Therefore, by Schur's Lemma, $M^{8k-2}$
admits no $Sp(2)\times Sp(k)$-invariant almost complex structure.

The proof of part (ii) will be divided into the following five cases.
\vspace{3mm}

\noindent
\textbf{(1) The $\mathbf{(G, K)=(SO(5)\times SO(5), \triangle SO(5))}$ case, \bm{$(g, m_1, m_2)=(4, 2, 2)$}:}

In this case, the isoparametric hypersurface $M^8$ is diffeomorphic to $SO(5)/T^2$, where $T^2$ is the maximal torus in $SO(5)$. Hence $M^8$ admits an $SO(5)$-invariant complex structure(cf. \cite{Wan54} and \cite{BH58}).
\vspace{3mm}

\noindent
\textbf{(2) The $\mathbf{(G, K)=(SO(k+2), SO(2)\times SO(k))}$ case, \bm{$(g, m_1, m_2)=(4, 1, k-2)$}:}

In this case, $(\mathfrak{g}, \mathfrak{k})=(\mathfrak{o}(2+k), \mathfrak{o}(2)+\mathfrak{o}(k))$, and we have the following restricted root decomposition
\begin{eqnarray*}
\mathfrak{k}&=&\mathfrak{k}_0\oplus \mathfrak{k}_{\alpha_1} \oplus \mathfrak{k}_{\alpha_2} \oplus \mathfrak{k}_{\alpha_1+\alpha_2} \oplus \mathfrak{k}_{\alpha_1-\alpha_2},\nonumber\\
\mathfrak{p}&=&\mathfrak{a}\oplus \mathfrak{p}_{\alpha_1} \oplus \mathfrak{p}_{\alpha_2} \oplus \mathfrak{p}_{\alpha_1+\alpha_2} \oplus \mathfrak{p}_{\alpha_1-\alpha_2},
\end{eqnarray*}
where
\begin{eqnarray*}
\mathfrak{k}&=&\left\{\begin{pmatrix}
	T_1 & 0\\
	0 & T_2
\end{pmatrix}~|~T_1\in \mathfrak{o}(2), T_2\in \mathfrak{o}(k)
\right\},\nonumber  \\
\mathfrak{p}&=&\left\{\begin{pmatrix}
	0 & Y\\
	-Y^T & 0
\end{pmatrix}~|~Y\in M(2, k; \mathbb{R})
\right\},\nonumber  \\
\mathfrak{k}_0&=&\left\{\begin{pmatrix}
	0 & 0 & 0\\
	0 & 0 & 0 \\
	0 & 0 & X
\end{pmatrix}~|~X\in \mathfrak{o}(k-2)
\right\},\nonumber \\
\mathfrak{k}_{\alpha_1}&=&\left\{\begin{pmatrix}
	0 & 0 & 0\\
	0 & 0 & X \\
	0 & -X^T & 0
\end{pmatrix}~|~X=
\begin{pmatrix}
	a \\
	0
\end{pmatrix}, a\in \mathbb{R}^{k-2}
\right\},\nonumber \\
\mathfrak{k}_{\alpha_2}&=&\left\{\begin{pmatrix}
	0 & 0 & 0\\
	0 & 0 & X \\
	0 & -X^T & 0
\end{pmatrix}~|~X=
\begin{pmatrix}
	0 \\
	b
\end{pmatrix}, b\in \mathbb{R}^{k-2}
\right\},\nonumber \\
\mathfrak{k}_{\alpha_1+\alpha_2}&=&\left\{\begin{pmatrix}
	\Lambda & 0 & 0\\
	0 & -\Lambda& 0 \\
	0 & 0 & 0
\end{pmatrix}~|~\Lambda=
\begin{pmatrix}
	0 & \lambda \\
	-\lambda & 0
\end{pmatrix}, \lambda\in \mathbb{R}
\right\},\nonumber \\
\mathfrak{k}_{\alpha_1-\alpha_2}&=&\left\{\begin{pmatrix}
	\Lambda & 0 & 0\\
	0 & \Lambda& 0 \\
	0 & 0 & 0
\end{pmatrix}~|~\Lambda=
\begin{pmatrix}
	0 & \lambda \\
	-\lambda & 0
\end{pmatrix}, \lambda\in \mathbb{R}
\right\},\nonumber \\
\mathfrak{a}&=&\left\{H=H(\xi_1, \xi_2)=\begin{pmatrix}
	0 &  \xi & 0\\
	-\xi^T & 0  &0 \\
	0 &0 &0
\end{pmatrix}~|~\xi=
\begin{pmatrix}
	\xi_1 & 0 \\
	0 & \xi_2
\end{pmatrix}, \xi_1, \xi_2\in \mathbb{R}
\right\},\nonumber \\
\mathfrak{p}_{\alpha_1}&=&\left\{\begin{pmatrix}
	0 &  0 &Y\\
	0 & 0  &0 \\
	-Y^T &0 & 0
\end{pmatrix}~|~Y=\begin{pmatrix}
	a \\
	0
\end{pmatrix},
a\in \mathbb{R}^{k-2}
\right\},\nonumber\\
\mathfrak{p}_{\alpha_2}&=&\left\{\begin{pmatrix}
	0 &  0 &Y\\
	0 & 0  &0 \\
	-Y^T &0 & 0
\end{pmatrix}~|~Y=\begin{pmatrix}
	0 \\
	b
\end{pmatrix},
b\in \mathbb{R}^{k-2}
\right\},\nonumber\\
\mathfrak{p}_{\alpha_1+\alpha_2}&=&\left\{\begin{pmatrix}
	0 &  \Lambda & 0\\
	-\Lambda^T & 0  &0 \\
	0 &0 & 0
\end{pmatrix}~|~\Lambda=
\begin{pmatrix}
	0 & \lambda \\
	-\lambda & 0
\end{pmatrix}, \lambda\in \mathbb{R}
\right\},\nonumber\\
\mathfrak{p}_{\alpha_1-\alpha_2}&=&\left\{\begin{pmatrix}
	0 &  \Lambda & 0\\
	-\Lambda^T & 0  &0 \\
	0 &0 & 0
\end{pmatrix}~|~\Lambda=
\begin{pmatrix}
	0 & \lambda \\
	\lambda & 0
\end{pmatrix}, \lambda\in \mathbb{R}
\right\},
\end{eqnarray*}
By definition, for $i=1, 2$, $\alpha_i(H)=\xi_i$. For $0<t<\frac{\pi}{4}$ and $H=H(\cos t, \sin t)$, the principal orbit $Ad(K)H$ is diffeomorphic to $K/K_0$, where
$$K_0=\left\{
\begin{pmatrix}
\pm I_4 & 0 \\
0 & T
\end{pmatrix}~|~
T\in SO(k-2)\right\}=\mathbb{Z}_2\times SO(k-2).$$
For the homogeneous isoparametric hypersurface $M=Ad(K)H\cong K/K_0$, we have the reductive decomposition $\mathfrak{k}=\mathfrak{k}_0\oplus\mathfrak{m}$, where $\mathfrak{m}=\mathfrak{k}_{\alpha_1+\alpha_2}\oplus \mathfrak{k}_{\alpha_1-\alpha_2}\oplus \mathfrak{k}_{\alpha_1}\oplus \mathfrak{k}_{\alpha_2}$. Moreover, the isotropy representations of $K_0$ on $\mathfrak{k}_{\alpha_1\pm\alpha_2}$ are trivial. The isotropy representations on $\mathfrak{k}_{\alpha_1}$ and $\mathfrak{k}_{\alpha_2}$ are irreducible and equivalent. Let $\{E_i\}$ be the standard basis of $\mathbb{R}^{k-2}$. For convenience, define a basis of $\mathfrak{k}_{\alpha_1}$ for $a\in \mathbb{R}^{k-2}$ corresponding to the standard basis $\{E_i\}$ of $\mathbb{R}^{k-2}$, denoted by $\{a_i\}$. Similarly, define a basis of $\mathfrak{k}_{\alpha_2}$ for $b\in \mathbb{R}^{k-2}$ corresponding to the standard basis $\{E_i\}$ of $\mathbb{R}^{k-2}$, denoted by $\{b_i\}$.
Choose a vector in $\mathfrak{k}_{\alpha_1+\alpha_2}$ corresponding to $\Lambda=\begin{pmatrix}
0 & 1     \\
-1 & 0
\end{pmatrix}$, denoted by $\Lambda_-$. Similarly, choose a vector in $\mathfrak{k}_{\alpha_1-\alpha_2}$ corresponding to $\Lambda=\begin{pmatrix}
0 & 1     \\
-1 & 0
\end{pmatrix}$, denoted by $\Lambda_+$.
By Schur's Lemma, for any given $SO(2)\times SO(k)$-invariant almost complex structure $J$, there exist constants $\lambda_1, \lambda_2$ and non-zero constants $\mu_1, \mu_2$, such that
\begin{eqnarray*}
Ja_i=\lambda_1a_i-\frac{1+\lambda_1^2}{\mu_1}b_i, &&Jb_i=\mu_1a_i-\lambda_1b_i, 1\leq i\leq k-2,\\
J\Lambda_-=\lambda_2\Lambda_--\frac{1+\lambda_2^2}{\mu_2}\Lambda_+, &&J\Lambda_+=\mu_2\Lambda_--\lambda_2\Lambda_+.
\end{eqnarray*}

Now, we are in a position to determine the integrability conditions of $J$.

For $u, v\in \mathfrak{k}_{\alpha_1}$, $N(u, v)=0$.

For $u\in \mathfrak{k}_{\alpha_1}, v\in \mathfrak{k}_{\alpha_2}$, $N(u, v)=0$.

For $u\in \mathfrak{k}_{\alpha_1}, v\in \mathfrak{k}_{\alpha_1+\alpha_2}$, $N(u, v)=0$  if and only if $$-\frac{1+\lambda_1^2}{\mu_1}(\lambda_2+\frac{1+\lambda_2^2}{\mu_2})+\lambda_1(\mu_1+\frac{1+\lambda_1^2}{\mu_1})+\mu_1(\lambda_2+\frac{1+\lambda_2^2}{\mu_2})=0$$
and $$-\lambda_1(\lambda_2+\frac{1+\lambda_2^2}{\mu_2})-(\lambda_1^2+\frac{(1+\lambda_1^2)^2}{\mu_1^2})-\lambda_1(\lambda_2+\frac{1+\lambda_2^2}{\mu_2})+1=0.$$

For $u\in \mathfrak{k}_{\alpha_1}, v\in \mathfrak{k}_{\alpha_1-\alpha_2}$, $N(u, v)=0$  if and only if
$$-\frac{1+\lambda_1^2}{\mu_1}(\lambda_2+\mu_2)-\lambda_1(\mu_1+\frac{1+\lambda_1^2}{\mu_1})+\mu_1(\lambda_2+\mu_2)=0$$ and $$-\lambda_1(\lambda_2+\mu_2)+(\lambda_1^2+\frac{(1+\lambda_1^2)^2}{\mu_1^2})-\lambda_1(\lambda_2+\mu_2)-1=0.$$

For $u, v\in \mathfrak{k}_{\alpha_2}$, $N(u, v)=0$.

For $u\in \mathfrak{k}_{\alpha_2}, v\in \mathfrak{k}_{\alpha_1+\alpha_2}$, $N(u, v)=0$ if and only if
$$-\lambda_1(\lambda_2+\frac{1+\lambda_2^2}{\mu_2})+(\lambda_1^2+\mu_1^2)-\lambda_1(\lambda_2+\frac{1+\lambda_2^2}{\mu_2})-1=0$$
and $$-\mu_1(\lambda_2+\frac{1+\lambda_2^2}{\mu_2})-\lambda_1(\mu_1+\frac{1+\lambda_1^2}{\mu_1})+\frac{1+\lambda_1^2}{\mu_1}(\lambda_2+\frac{1+\lambda_2^2}{\mu_2})=0.$$

For $u\in \mathfrak{k}_{\alpha_2}, v\in \mathfrak{k}_{\alpha_1-\alpha_2}$, $N(u, v)=0$ if and only if
$$-\lambda_1(\lambda_2+\mu_2)-(\lambda_1^2+\mu_1^2)-\lambda_1(\lambda_2+\mu_2)+1=0$$
and
$$-\mu_1(\lambda_2+\mu_2)+\lambda_1(\mu_1+\frac{1+\lambda_1^2}{\mu_1})+\frac{1+\lambda_1^2}{\mu_1}(\lambda_2+\mu_2)=0.$$

For $u, v\in \mathfrak{k}_{\alpha_1+\alpha_2}$, $N(u, v)=0$.

For $u \in \mathfrak{k}_{\alpha_1+\alpha_2}, v\in \mathfrak{k}_{\alpha_1-\alpha_2}$, $N(u, v)=0$.

For $u, v\in \mathfrak{k}_{\alpha_1-\alpha_2}$, $N(u, v)=0$.

Particularly, if $\lambda_1=\lambda_2=0$, then $J$ is integrable  if and only if $\mu_1^2=1$. It follows that the homogeneous isoparametric hypersurface in this case admits at least one complex structure.
\vspace{3mm}

\noindent
\textbf{(3) The $\mathbf{(G, K)=(SU(k+2), S(U(2)\times U(k))}$ case, \bm{$(g, m_1, m_2)=(4, 2, 2k-3)$}:}

In this case, equivalently, we use the symmetric pair $(G, K)=(U(k+2), U(2)\times U(k))$ and $(\mathfrak{g}, \mathfrak{k})=(\mathfrak{u}(k+2), \mathfrak{u}(2)+\mathfrak{u}(k)).$ The restricted root decomposition is given by
\begin{eqnarray*}
\mathfrak{k}&=&\mathfrak{k}_0\oplus \mathfrak{k}_{\alpha_1} \oplus \mathfrak{k}_{\alpha_2} \oplus \mathfrak{k}_{2\alpha_1} \oplus \mathfrak{k}_{2\alpha_2}\oplus \mathfrak{k}_{\alpha_1+\alpha_2} \oplus \mathfrak{k}_{\alpha_1-\alpha_2},\nonumber\\
\mathfrak{p}&=&\mathfrak{a}\oplus \mathfrak{p}_{\alpha_1} \oplus \mathfrak{p}_{\alpha_2} \oplus \mathfrak{p}_{2\alpha_1} \oplus \mathfrak{p}_{2\alpha_2}\oplus \mathfrak{p}_{\alpha_1+\alpha_2} \oplus \mathfrak{p}_{\alpha_1-\alpha_2},
\end{eqnarray*}
where
\begin{eqnarray*}
\mathfrak{k}&=&\left\{\begin{pmatrix}
	T_1 & 0\\
	0 & T_2
\end{pmatrix}~|~T_1\in \mathfrak{u}(2), T_2\in \mathfrak{u}(k)
\right\},\nonumber  \\
\mathfrak{p}&=&\left\{\begin{pmatrix}
	0 & Y\\
	-\overline{Y}^T & 0
\end{pmatrix}~|~Y\in M(2, k; \mathbb{C})
\right\},\nonumber  \\
\mathfrak{k}_0&=&\left\{\begin{pmatrix}
	\Lambda & 0 & 0\\
	0 & \Lambda & 0 \\
	0 & 0 & T_3
\end{pmatrix}~|~\Lambda=\begin{pmatrix}
	u_{11} & 0 \\
	0 & u_{22}
\end{pmatrix}, u_{11}, u_{22}\in \mathrm{Im} \mathbb{C}, T_3\in \mathfrak{u}(k-2)
\right\},\nonumber \\
\mathfrak{k}_{\alpha_1}&=&\left\{\begin{pmatrix}
	0 & 0 & 0\\
	0 & 0 & X \\
	0 & -\overline{X}^T & 0
\end{pmatrix}~|~X=
\begin{pmatrix}
	a \\
	0
\end{pmatrix}, a\in \mathbb{C}^{k-2}
\right\},\nonumber \\
\mathfrak{k}_{\alpha_2}&=&\left\{\begin{pmatrix}
	0 & 0 & 0\\
	0 & 0 & X \\
	0 & -\overline{X}^T & 0
\end{pmatrix}~|~X=
\begin{pmatrix}
	0 \\
	b
\end{pmatrix}, b\in \mathbb{C}^{k-2}
\right\},\nonumber \\
\mathfrak{k}_{2\alpha_1}&=&\left\{\begin{pmatrix}
	\Lambda & 0 & 0\\
	0 & -\Lambda& 0 \\
	0 & 0 & 0
\end{pmatrix}~|~\Lambda=
\begin{pmatrix}
	\eta_{11} & 0 \\
	0 & 0
\end{pmatrix}, \eta_{11}\in \mathrm{Im} \mathbb{C}
\right\},\nonumber \\
\mathfrak{k}_{2\alpha_2}&=&\left\{\begin{pmatrix}
	\Lambda & 0 & 0\\
	0 & -\Lambda& 0 \\
	0 & 0 & 0
\end{pmatrix}~|~\Lambda=
\begin{pmatrix}
	0 & 0 \\
	0 & \eta_{22}
\end{pmatrix}, \eta_{22}\in \mathrm{Im} \mathbb{C}
\right\},\nonumber \\
\mathfrak{k}_{\alpha_1+\alpha_2}&=&\left\{\begin{pmatrix}
	\Lambda & 0 & 0\\
	0 & -\Lambda& 0 \\
	0 & 0 & 0
\end{pmatrix}~|~\Lambda=
\begin{pmatrix}
	0 & \eta_{12} \\
	-\overline{\eta}_{12} & 0
\end{pmatrix}, \eta_{12}\in \mathbb{C}
\right\},\nonumber \\
\mathfrak{k}_{\alpha_1-\alpha_2}&=&\left\{\begin{pmatrix}
	\Lambda & 0 & 0\\
	0 & \Lambda& 0 \\
	0 & 0 & 0
\end{pmatrix}~|~\Lambda=
\begin{pmatrix}
	0 & \eta_{12} \\
	-\overline{\eta}_{12} & 0
\end{pmatrix}, \eta_{12}\in \mathbb{C}
\right\},\nonumber \\
\mathfrak{a}&=&\left\{H=H(\xi_1, \xi_2)=\begin{pmatrix}
	0 &  \xi \\
	-\xi^T & 0 \\
\end{pmatrix}~|~\xi=
\begin{pmatrix}
	\xi_1 & 0 & & 0\\
	0 & \xi_2 & & 0
\end{pmatrix}, \xi_1, \xi_2\in \mathbb{R}
\right\},\nonumber \\
\mathfrak{p}_{\alpha_1}&=&\left\{\begin{pmatrix}
	0 &  0 &Y\\
	0 & 0  &0 \\
	-\overline{Y}^T &0 & 0
\end{pmatrix}~|~Y=\begin{pmatrix}
	a \\
	0
\end{pmatrix},
a\in \mathbb{C}^{k-2}
\right\},\nonumber\\
\mathfrak{p}_{\alpha_2}&=&\left\{\begin{pmatrix}
	0 &  0 &Y\\
	0 & 0  &0 \\
	-\overline{Y}^T &0 & 0
\end{pmatrix}~|~Y=\begin{pmatrix}
	0 \\
	b
\end{pmatrix},
b\in \mathbb{C}^{k-2}
\right\},\nonumber\\
\mathfrak{p}_{2\alpha_1}&=&\left\{\begin{pmatrix}
	0 & \Lambda &0\\
	-\overline{\Lambda}^T & 0  &0 \\
	0 &0 & 0
\end{pmatrix}~|~\Lambda=\begin{pmatrix}
	\eta_{11}& 0\\
	0& 0
\end{pmatrix},
\eta_{11}\in \mathrm{Im} \mathbb{C}
\right\},\nonumber\\
\mathfrak{p}_{2\alpha_2}&=&\left\{\begin{pmatrix}
	0 & \Lambda &0\\
	-\overline{\Lambda}^T & 0  &0 \\
	0 &0 & 0
\end{pmatrix}~|~\Lambda=\begin{pmatrix}
	0& 0\\
	0& \eta_{22}
\end{pmatrix},
\eta_{22}\in \mathrm{Im} \mathbb{C}
\right\},\nonumber\\
\mathfrak{p}_{\alpha_1+\alpha_2}&=&\left\{\begin{pmatrix}
	0 &  \Lambda & 0\\
	-\overline{\Lambda}^T & 0  &0 \\
	0 &0 & 0
\end{pmatrix}~|~\Lambda=
\begin{pmatrix}
	0 & \eta_{12} \\
	-\overline{\eta}_{12} & 0
\end{pmatrix}, \eta_{12}\in \mathbb{C}
\right\},\nonumber\\
\mathfrak{p}_{\alpha_1-\alpha_2}&=&\left\{\begin{pmatrix}
	0 &  \Lambda & 0\\
	-\overline{\Lambda}^T & 0  &0 \\
	0 &0 & 0
\end{pmatrix}~|~\Lambda=
\begin{pmatrix}
	0 & \eta_{12} \\
	\overline{\eta}_{12} & 0
\end{pmatrix}, \eta_{12}\in \mathbb{C}
\right\}.
\end{eqnarray*}
For $0<t<\frac{\pi}{4}$, and $H=H(\cos t, \sin t)$, the principal orbit $Ad(K)H$ is diffeomorphic to $K/K_0$, where
\begin{eqnarray*}
K_0&=&\left\{
\begin{pmatrix}
	\Lambda & 0 & 0\\
	0 & \Lambda& 0\\
	0 & 0 & T
\end{pmatrix}~|~
\Lambda=
\begin{pmatrix}
	e^{\sqrt{-1}s} & 0 \\
	0 & e^{\sqrt{-1}t}
\end{pmatrix}, s, t\in \mathbb{R}, T\in U(k-2)\right\}\nonumber\\
&=&\Delta(U(1)\times U(1))\times U(k-2).
\end{eqnarray*}
For the homogeneous isoparametric hypersurface $M=Ad(K)H\cong K/K_0$, we have the reductive decomposition $\mathfrak{k}=\mathfrak{k}_0\oplus\mathfrak{m}$, where $\mathfrak{m}=\mathfrak{k}_{\alpha_1+\alpha_2}\oplus \mathfrak{k}_{\alpha_1-\alpha_2}\oplus \mathfrak{k}_{\alpha_1}\oplus \mathfrak{k}_{\alpha_2}\oplus \mathfrak{k}_{2\alpha_1}\oplus \mathfrak{k}_{2\alpha_2}$. Moreover, the isotropy representations of $K_0$ on $\mathfrak{k}_{2\alpha_1}$ and $\mathfrak{k}_{2\alpha_2}$ are trivial. The isotropy representations on $\mathfrak{k}_{\alpha_1}$ and $\mathfrak{k}_{\alpha_2}$ are irreducible but not equivalent. The isotropy representations on $\mathfrak{k}_{\alpha_1\pm \alpha_2}$ are irreducible and equivalent.
Let $\{E_i\}$ be the standard basis of $\mathbb{R}^{k-2}$. Then $\{E_i, \sqrt{-1}E_i\}$ forms a real basis of   $\mathbb{C}^{k-2}$. For convenience, define a basis of $\mathfrak{k}_{\alpha_1}$ for $a\in \mathbb{C}^{k-2}$ corresponding to the real basis $\{E_i, \sqrt{-1}E_i\}$ of $\mathbb{C}^{k-2}$, denoted by $\{a_i, A_i\}$. Similarly, define a basis of $\mathfrak{k}_{\alpha_2}$ for $b\in \mathbb{C}^{k-2}$ corresponding to the real basis $\{E_i, \sqrt{-1}E_i\}$ of $\mathbb{C}^{k-2}$, denoted by $\{b_i, B_i\}$.
Choose a vector in $\mathfrak{k}_{2\alpha_1}$ corresponding to $\eta_{11}=\sqrt{-1}$, denoted by $\eta_{11}$. Similarly, choose a vector in $\mathfrak{k}_{2\alpha_2}$ corresponding to $\eta_{22}=\sqrt{-1}$, denoted by $\eta_{22}$. Choose a basis $e_1, e_2$ in $\mathfrak{k}_{\alpha_1+\alpha_2}$ corresponding to $\eta_{12}=\begin{pmatrix}
0 & 1 \\
-1 & 0
\end{pmatrix},
\begin{pmatrix}
0 & \sqrt{-1} \\
\sqrt{-1} & 0
\end{pmatrix},$ respectively.
Similarly, choose a basis $f_1, f_2$ in $\mathfrak{k}_{\alpha_1-\alpha_2}$ corresponding to $\eta_{12}=\begin{pmatrix}
0 & 1 \\
-1 & 0
\end{pmatrix},
\begin{pmatrix}
0 & \sqrt{-1} \\
\sqrt{-1} & 0
\end{pmatrix},$ respectively.

In this case, for simplicity, we only consider the following special almost complex structure and its integrability condition. For non-zero constants $\mu_1, \mu_2$, define a $U(2)\times U(k)$-invariant almost complex structure $J$ by
\begin{eqnarray*}
Ja_i=A_i, ~~JA_i=-a_i, &&Jb_i=B_i, ~~ JB_i=-b_i, \quad 1\leq i\leq k-2,\\
J\eta_{11}=-\frac{1}{\mu_1}\eta_{22}, &&J\eta_{22}=\mu_{1}\eta_{11},\\
Je_1=-\frac{1}{\mu_2}f_2, ~~ Jf_2=\mu_2e_1, &&Je_2=\frac{1}{\mu_2}f_1,  ~~ Jf_1=-\mu_2e_2.
\end{eqnarray*}
Now, we are in a position to determine the integrability condition of $J$. By the definition of $J$ and the the decomposition of $\mathfrak{m}$, we can determine the Nijenhuis tensor $N$. For instance, $N(\eta_{11}, e_1)=0$ if and only if $\mu_2^2=1$. In fact, $J$ is integrable if and only if $\mu_2^2=1$, and the homogeneous isoparametric hypersurface in this case admits a complex structure.
\vspace{3mm}

\noindent
\textbf{(4) The $\mathbf{(G, K)=(SO(10), U(5))}$ case, \bm{$(g, m_1, m_2)=(4, 4, 5)$}:}

In this case, $U(5)$ is embedded in $SO(10)$ by
\begin{eqnarray*}
U(5) &\hookrightarrow& SO(10),\nonumber\\
a+\sqrt{-1}b &\mapsto& \begin{pmatrix}
	a & b\\
	-b & a
\end{pmatrix}.
\end{eqnarray*}
Consequently, $\mathfrak{u}(5)=\{X+\sqrt{-1} Y~|~X+X^T=0, Y=Y^T, X, Y\in M(5, 5, \mathbb{R})\}$ is identified with $\mathfrak{k}=
\left\{\begin{pmatrix}
X & Y\\
-Y & X
\end{pmatrix}~|~ X+X^T=0, Y=Y^T, X, Y\in M(5, 5, \mathbb{R})\right\}\subset \mathfrak{o}(10)$ by $X+\sqrt{-1}Y\mapsto \begin{pmatrix}
X & Y\\
-Y & X
\end{pmatrix}.$

Write $X=\begin{pmatrix}
X_{11} & X_{12} & X_{13}\\
-X_{12}^T & X_{22}& X_{23}\\
-X_{13}^T & -X_{23}^T& 0
\end{pmatrix},$
and $Y=\begin{pmatrix}
Y_{11} & Y_{12} & Y_{13}\\
Y_{12}^T & Y_{22}& Y_{23}\\
Y_{13}^T & Y_{23}^T& Y_{33}
\end{pmatrix},$
where $X_{11}, X_{12}, X_{22}$, $Y_{11}, Y_{12}, Y_{22}\in M(2, 2, \mathbb{R}),$
$X_{13}, X_{23}, Y_{13}, Y_{23}\in M(2, 1, \mathbb{R}),$ and $Y_{33}\in \mathbb{R}.$
In this case, $(\mathfrak{g}, \mathfrak{k})=(\mathfrak{o}(10), \mathfrak{u}(5))$. Then
$$\mathfrak{p}=\left\{\begin{pmatrix}
A & B\\
B & -A
\end{pmatrix}~|~ A, B\in \mathfrak{o}(5)\right\}\subset \mathfrak{o}(10).$$
Write $A=\begin{pmatrix}
A_{11} & A_{12} & A_{13}\\
-A_{12}^T & A_{22}& A_{23}\\
-A_{13}^T & -A_{23}^T& 0
\end{pmatrix},$
and $B=\begin{pmatrix}
B_{11} &B_{12} & B_{13}\\
-B_{12}^T & B_{22}& B_{23}\\
-B_{13}^T & -B_{23}^T& 0
\end{pmatrix},$
where $A_{11}, A_{12}, A_{22}$, $B_{11}, B_{12}, B_{22}\in M(2, 2, \mathbb{R}),$
and $A_{13}, A_{23}, B_{13}, B_{23}\in M(2, 1, \mathbb{R}).$ Define $J_0=\begin{pmatrix}
0 & 1\\
-1&0
\end{pmatrix}.$
The restricted root decomposition is given by
\begin{eqnarray*}
\mathfrak{k}&=&\mathfrak{k}_0\oplus \mathfrak{k}_{\alpha_1} \oplus \mathfrak{k}_{\alpha_2} \oplus \mathfrak{k}_{2\alpha_1} \oplus \mathfrak{k}_{2\alpha_2}\oplus \mathfrak{k}_{\alpha_1+\alpha_2} \oplus \mathfrak{k}_{\alpha_1-\alpha_2},\nonumber\\
\mathfrak{p}&=&\mathfrak{a}\oplus \mathfrak{p}_{\alpha_1} \oplus \mathfrak{p}_{\alpha_2} \oplus \mathfrak{p}_{2\alpha_1} \oplus \mathfrak{p}_{2\alpha_2}\oplus \mathfrak{p}_{\alpha_1+\alpha_2} \oplus \mathfrak{p}_{\alpha_1-\alpha_2},
\end{eqnarray*}
where
\begin{eqnarray*}
\mathfrak{k}_0&:& X=\begin{pmatrix}
	\xi_1J_0 & 0 &0\\
	0 & \xi_2J_0 &0\\
	0 & 0 &0
\end{pmatrix},
Y_{11}, Y_{22}\in \mathrm{Span}\left\{\begin{pmatrix}
	0 & 1\\
	1& 0
\end{pmatrix},
\begin{pmatrix}
	1 & 0\\
	0& -1
\end{pmatrix}\right\}, Y_{33}\in \mathbb{R}, \nonumber \\
\mathfrak{k}_{\alpha_1}&:&X_{13}, Y_{13}\in M(2, 1, \mathbb{R}),\nonumber \\
\mathfrak{k}_{\alpha_2}&:&X_{23}, Y_{23}\in M(2, 1, \mathbb{R}),\nonumber \\
\mathfrak{k}_{2\alpha_1}&:&Y_{11}\in \mathbb{R}I_2,\nonumber \\
\mathfrak{k}_{2\alpha_2}&:&Y_{22}\in \mathbb{R}I_2,\nonumber \\
\mathfrak{k}_{\alpha_1+\alpha_2}&:&X_{12}\in \mathrm{Span}\left\{ \begin{pmatrix}
	1 & 0\\
	0& -1
\end{pmatrix},
\begin{pmatrix}
	0 & 1\\
	1& 0
\end{pmatrix}\right\},
Y_{12}\in \mathrm{Span}\left\{I_2,
J_0 \right\},\nonumber \\
\mathfrak{k}_{\alpha_1-\alpha_2}&:&X_{12}\in \mathrm{Span}\left\{I_2, J_0\right\},
Y_{12}\in \mathrm{Span}\left\{\begin{pmatrix}
	0 & 1\\
	1& 0
\end{pmatrix},
\begin{pmatrix}
	1 & 0\\
	0& -1
\end{pmatrix}\right\},\nonumber \\
\mathfrak{a}&:&A=H(\xi_1, \xi_2)=\begin{pmatrix}
	\xi_1J_0 & 0& 0 \\
	0 & \xi_2J_0& 0 \\
	0& 0 & 0
\end{pmatrix}, \xi_1, \xi_2\in \mathbb{R},\nonumber \\
\mathfrak{p}_{\alpha_1}&:&A_{13}, B_{13}\in M(2, 1, \mathbb{R}),\nonumber\\
\mathfrak{p}_{\alpha_2}&:&A_{23}, B_{23}\in M(2, 1, \mathbb{R}),\nonumber\\
\mathfrak{p}_{2\alpha_1}&:&B_{11}\in \mathbb{R}J_0,\nonumber\\
\mathfrak{p}_{2\alpha_2}&:&B_{22}\in \mathbb{R}J_0,\nonumber\\
\mathfrak{p}_{\alpha_1+\alpha_2}&:&A_{12}\in \mathrm{Span}\left\{\begin{pmatrix}
	0 & 1\\
	1& 0
\end{pmatrix},
\begin{pmatrix}
	1 & 0\\
	0& -1
\end{pmatrix}\right\},
B_{12}\in \mathrm{Span}\left\{I_2,
J_0\right\},\nonumber\\
\mathfrak{p}_{\alpha_1-\alpha_2}&:&A_{12}\in \mathrm{Span}\left\{I_2,
J_0\right\}, B_{12}\in \mathrm{Span}\left\{\begin{pmatrix}
	0 & 1\\
	1& 0
\end{pmatrix},
\begin{pmatrix}
	1 & 0\\
	0& -1
\end{pmatrix}\right\}.
\end{eqnarray*}
Under the identification $\mathfrak{u}(5)\cong \mathfrak{k}$, we have $\mathfrak{k}_0\cong \mathfrak{su}(2)+\mathfrak{su}(2)+\mathfrak{u}(1)\subset \mathfrak{u}(5)$. For $0<t<\frac{\pi}{4}$, and $H=H(\cos t, \sin t)$, the principal orbit $Ad(K)H$ is diffeomorphic to $K/K_0$, where
$K_0\cong SU(2)\times SU(2)\times U(1).$
For the homogeneous isoparametric hypersurface $M=Ad(K)H\cong U(5)/(SU(2)\times SU(2)\times U(1))$, the reductive decomposition is given by $\mathfrak{k}=\mathfrak{k}_0\oplus\mathfrak{m}$, where $\mathfrak{m}=\mathfrak{k}_{\alpha_1+\alpha_2}\oplus \mathfrak{k}_{\alpha_1-\alpha_2}\oplus \mathfrak{k}_{\alpha_1}\oplus \mathfrak{k}_{\alpha_2}\oplus \mathfrak{k}_{2\alpha_1}\oplus \mathfrak{k}_{2\alpha_2}$. Moreover, the isotropy representations of $K_0$ on $\mathfrak{k}_{2\alpha_1}$ and $\mathfrak{k}_{2\alpha_2}$ are trivial. The isotropy representations on $\mathfrak{k}_{\alpha_1}$ and $\mathfrak{k}_{\alpha_2}$ are irreducible but not equivalent.
The isotropy representations of $K_0$ on $\mathfrak{k}_{\alpha_1\pm\alpha_2}$ are given by
\begin{eqnarray*}
&&\begin{pmatrix}
	a &  &\\
	& b&\\
	&&e^{\sqrt{-1}\theta}
\end{pmatrix}\cdot
\begin{pmatrix}
	0 & X_{12}+\sqrt{-1}Y_{12}& 0\\
	-X_{12}^T+\sqrt{-1}Y_{12}^T& 0& 0\\
	0& 0& 0
\end{pmatrix}\nonumber\\
&=&\begin{pmatrix}
	0 & a(X_{12}+\sqrt{-1}Y_{12})\overline{b}^T& 0\\
	b(-X_{12}^T+\sqrt{-1}Y_{12}^T)\overline{a}^T& 0& 0\\
	0& 0& 0
\end{pmatrix}.
\end{eqnarray*}
It follows that the isotropy representations of $K_0$ on $\mathfrak{k}_{\alpha_1\pm\alpha_2}$ are irreducible. For convenience, we use $X_{12}+\sqrt{-1}Y_{12}$ to represent
$$\begin{pmatrix}
0 & X_{12}+\sqrt{-1}Y_{12}& 0\\
-X_{12}^T+\sqrt{-1}Y_{12}^T& 0& 0\\
0& 0& 0
\end{pmatrix}\in \mathfrak{k}_{\alpha_1\pm\alpha_2}.$$
Define
\begin{eqnarray*}
T&:& \mathfrak{k}_{\alpha_1+\alpha_2}\rightarrow \mathfrak{k}_{\alpha_1-\alpha_2}\nonumber\\
&&X_{12}+\sqrt{-1}Y_{12}\mapsto \sqrt{-1}X_{12}-Y_{12}.
\end{eqnarray*}
Thus $T$ is a homomorphism of the isotropy representations $K_0$ on $\mathfrak{k}_{\alpha_1\pm\alpha_2}$.

\begin{lem}
The isotropy representations of $K_0$ on $\mathfrak{k}_{\alpha_1\pm\alpha_2}$ are equivalent. Moreover, for any isomorphism $\varphi: \mathfrak{k}_{\alpha_1+\alpha_2}\rightarrow \mathfrak{k}_{\alpha_1-\alpha_2}$ between the isotropy representations of $K_0$ on $\mathfrak{k}_{\alpha_1\pm\alpha_2}$, there exists $c\in \mathbb{R}$ such that $\varphi=cT$.
\end{lem}

\begin{proof}
First, observe that $T$ is an isomorphism and that the isotropy representations of $K_0$ on $\mathfrak{k}_{\alpha_1\pm\alpha_2}$ are equivalent.
Next, the isotropy representations of $K_0$ on $\mathfrak{k}_{\alpha_1\pm\alpha_2}$ essentially induce irreducible representations of $(SU(2)\times SU(2))/\{(I_2, I_2), (-I_2, -I_2)\}\cong SO(4)$. Since any irreducible representation of $SO(4)$ must be absolutely irreducible, it follows from Schur's Lemma that there exists $c\in \mathbb{R}$ such that $\varphi=cT$.
\end{proof}
Using this lemma, it is possible to determine all $U(5)$-invariant almost complex structures on isoparametric hypersurface in this case. For simplicity, to finish the proof, we only consider the following particular $U(5)$-invariant almost complex structure and study the integrability condition.

Define a basis $g_1, g_2, h_1, h_2$ of $\mathfrak{k}_{\alpha_1}$ such that $g_1, g_2$ are corresponding to $X_{13}=(1, 0)^T, (0, 1)^T$, and $h_1, h_2$ are corresponding to $Y_{13}=(1, 0)^T, (0, 1)^T$, respectively. Similarly, choose a basis $k_1, k_2, L_1, L_2$ of $\mathfrak{k}_{\alpha_2}$, such that $k_1, k_2$ are corresponding to $X_{23}=(1, 0)^T, (0, 1)^T$, and $L_1, L_2$ are corresponding to $Y_{23}=(1, 0)^T, (0, 1)^T$, respectively. Choose a basis $e_1, e_2, e_3, e_4$ in $\mathfrak{k}_{\alpha_1+\alpha_2}$ corresponding to $X_{12}=\begin{pmatrix}
1 & 0 \\
0 & -1
\end{pmatrix},
\begin{pmatrix}
0 & 1 \\
1 & 0
\end{pmatrix},
Y_{12}=J_0, I_2,$ respectively.
Similarly, choose a basis $f_1, f_2, f_3, f_4$ in $\mathfrak{k}_{\alpha_1-\alpha_2}$ corresponding to $Y_{12}=\begin{pmatrix}
1 & 0 \\
0 & -1
\end{pmatrix},$
$\begin{pmatrix}
0 & 1 \\
1 & 0
\end{pmatrix}, X_{12}=-J_0, -I_2$, respectively.
Choose a vector $e\in \mathfrak{k}_{2\alpha_1}$ corresponding to $Y_{11}=I_2$, and a vector $f\in \mathfrak{k}_{2\alpha_2}$ corresponding to $Y_{22}=I_2.$
For non-zero constants $\mu_1, \mu_2$, define a $U(2)\times U(k)$-invariant almost complex structure $J$ by
\begin{eqnarray*}
Je=-\frac{1}{\mu_1}f, &~~Jf=\mu_1e,&~~ Je_i=-\frac{1}{\mu_2}f_i, ~~Jf_i=\mu_2e_i,~~ 1\leq i\leq 4,\\
Jg_1=h_1, &Jg_2=h_2,& Jh_1=-g_1,\quad\,\,  Jh_2=-g_2, \\
Jk_1=L_1,& Jk_2=L_2,& JL_1=-k_1, \quad\,\, JL_2=-k_2.
\end{eqnarray*}
Now, we are in a position to study the integrability condition of $J$. As in the $(SU(k+2), S(U(2)\times U(k))$ case, by the definition of $J$ and the decomposition of $\mathfrak{m}$, the Nijenhuis tensor $N$ can be determined. For instance, $N(e, e_1)=0$ if and only if $\mu_2^2=1$. Actually,
$J$ is integrable if and only if $\mu_2^2=1$, and the homogeneous isoparametric hypersurface in this case admits a complex structure.
\vspace{3mm}

\noindent
\textbf{(5) The $\mathbf{(G, K)=(E_6, U(1)\cdot Spin(10))}$ case, \bm{$(g, m_1, m_2)=(4, 6, 9)$}:}

In this case, the homogeneous isoparametric hypersurface $M$ is diffeomorphic to $U(1)\cdot Spin(10) / (U(1)\cdot SU(4))\cong Spin(10)/SU(4)$. According to Theorem VII of \cite{Wan54}, $Spin(10)/SU(4)$ admits an $Spin(10)$-invariant complex structure.
\end{proof}

\subsection{Invariant properties of complex structures constructed in Section \ref{M1} and \ref{M_+ with m=2,4}.}
In this section, we will discuss the invariant properties of the complex structure obtained in Section \ref{M1} and \ref{M_+ with m=2,4} based on Cartan's moving frame method and symmetric Clifford system, respectively. We start with the following fundamental lemma.
\begin{lem}
Let $G$ be a connected compact Lie group and $\varphi: G\rightarrow SO(n)$ be an orthogonal representation on $\mathbb{R}^n$. For $p\in S^{n-1}\subset \mathbb{R}^n$, consider the orbit $M^k=Gp$. Let $J: TM\rightarrow TM$ be an almost complex structure. Then $J$ is $G$-invariant if and only if for any $g\in G$
$$J_{gp}\circ \varphi(g)=\varphi(g)\circ J_p: T_pM\subset \mathbb{R}^n\rightarrow T_{gp}M\subset \mathbb{R}^n.$$
\end{lem}

As an application of the lemma above, we obtain
\begin{prop}\label{invariant property}
\begin{itemize}
\item[(i)] For $g=4, m=1$, the complex structure of the isoparametric hypersurface defined in (\ref{ new J}) is $SO(2)\times SO(l)$-invariant.
\item[(ii)] For $g=4, m=2$, the complex structure of the focal submanifold $M_+$ defined in (\ref{J2}) is $U(l)$-invariant but not $U(2)\times U(l)$-invariant.
\item[(iii)] For $g=4, m=4$ in the definite case, the complex structure of the focal submanifold $M_+$ defined in (\ref{J3}) is $Sp(l)$-invariant but not $Sp(2)\times Sp(l)$-invariant.
\end{itemize}
\end{prop}

\section{\textbf{Geometric Properties of Complex Structures}}
In this section, we will discuss the special metrics of complex structures related to isoparametric theory.
\subsection{\textbf{K\"{a}hler metric}}
According to M\"{u}nzner \cite{Mun81} and the cohomology theory of K\"{a}hler metrics, for $g=4$, the only isoparametric hypersurface that admits a K\"{a}hler structure is that with $(m_1, m_2)=(2,2)$. Therefore, in general, we cannot expect that the complex structure associated with isoparametric theory admits a K\"{a}hler metric.

\subsection{\textbf{Balanced metric}}
For a non-K\"{a}hler complex manifold $M^n$, it is interesting to determine whether $M$ admits a balanced metric. By definition, a Hermitian metric is called balanced if its K\"{a}hler form $\omega$ satisfies $d(\omega^{n-1})=0$.

\subsubsection{g=4, m=1}
Based on the results in Section \ref{M1}, given $0<t<\frac{\pi}{4}$, consider the isoparametric hypersurface $M_t$ for $g=4, m=1$ with the complex structure $J$ defined in Theorem \ref{m=1}. Define a natural Hermitian metric $ds^2$ by modifying the induced metric on different principal distributions ${\mathcal{D}}_1$, ${\mathcal{D}}_2$, ${\mathcal{D}}_3$, ${\mathcal{D}}_4$ with constants $\frac{\sin{t}+\cos{t}}{\sqrt{2}}, \cos{t}, \frac{\sin{t}-\cos{t}}{\sqrt{2}}, \sin{t}$, respectively. Then
\begin{prop}
The Hermitian manifold $(M_t, ds^2, J)$ defined above is not balanced.
\end{prop}
\begin{proof}
We only consider the case $(g, m_1, m_2)=(4, 1, 2)$. For general $m_2=l-2$, the proof is similar and omitted.
Using the unitary frame in the proof of Theorem \ref{m=1}, we find that it is an orthogonal frame with respect to $ds^2$.
Recall that the K\"{a}hler form $\alpha$ is defined by $\alpha(X, Y)=ds^2(JX, Y)$.
Using the same dual frame as in the proof of Theorem \ref{m=1}, we get that
$\alpha=(d\theta-\omega_{12})\wedge (d\theta-\omega_{12})+\omega_{13}\wedge \omega_{23}+\omega_{14}\wedge\omega_{24}$.
Then by the structure equation of orthogonal group and direct computations, $d(\alpha^2)=8\omega_{14}\wedge\omega_{24}\wedge\omega_{13}\wedge \omega_{23}\bigwedge d\theta\neq 0.$ Therefore, $(M_t, J, ds^2)$ is not balanced.
\end{proof}

\subsubsection{g=4, m=2, 4}
\begin{prop}
\begin{itemize}
\item[(i)] For $g=4, m=2$, the focal submanifold $M_+$ with the complex structure $J$ defined in (\ref{J2}) cannot admit any balanced metric. 
\item[(ii)] For $g=4, m=4$, the focal submanifold $M_+$ with the complex structure $J$ defined in (\ref{J3}) cannot admit any balanced metric.
\end{itemize}
\end{prop}
\begin{proof}
(i). In the case $g=4, m=2, l=2k$, the focal submanifold $M_+$ is diffeomorphic to $U(k)/U(k-2)$. It follows that $M_+$ is an M-manifold(cf. \cite{Wan54}). By Proposition \ref{invariant property}, the focal submanifold $M_+$ with the complex structure $J$ defined in Section \ref{M+2} is $U(k)$-invariant. Then by Theorem 2 of \cite{Pod18}, $(M_+, J)$ in this case cannot admit any balanced metric.

The proof for part (ii) is similar and omitted.
\end{proof}


\end{document}